\documentclass[11pt,reqno]{amsart}
\usepackage{amsmath}
\usepackage{amsthm}
\usepackage{graphicx}
\usepackage{amsfonts}
\usepackage{amssymb}
\usepackage{hyperref}
\usepackage{bm}
\usepackage{color}
\usepackage[margin=1.1in]{geometry}
\usepackage{enumerate}
\usepackage{mathrsfs}

\usepackage[numbers,square]{natbib}
\usepackage{mathtools}

\numberwithin{equation}{section}
\newtheorem{theorem}{Theorem}[section]
\newtheorem{lemma}[theorem]{Lemma}
\newtheorem{proposition}[theorem]{Proposition}

\usepackage{amsthm}
\theoremstyle{plain}
\newtheorem*{assumption*}{\assumptionnumber}
\providecommand{\assumptionnumber}{}
\makeatletter
\newenvironment{assumption}[1]
 {%
  \renewcommand{\assumptionnumber}{Assumption #1}%
  \begin{assumption*}%
  \protected@edef\@currentlabel{\textbf{\textup {#1}}}%
 }
 {%
  \end{assumption*}
 }

\theoremstyle{definition}
\newtheorem{example}[theorem]{Example}
\newtheorem{definition}[theorem]{Definition}
\newtheorem{remark}[theorem]{Remark}

\def\E{{\mathbb E}}
\def\R{{\mathbb R}}
\def\N{{\mathbb N}}
\def\P{{\mathbb P}}

\makeindex

\def\E{{\mathbb E}}
\def\R{{\mathbb R}}
\def\N{{\mathbb N}}
\def\P{{\mathbb P}}
\def\D{D}

\newcommand{\ent}{ \textup {H}}

\makeindex
\begin{document}
\title[Propagation of Chaos for Conditional McKean--Vlasov Equations]{Regularity and Propagation of Chaos for conditional McKean--Vlasov equations}
\author{Manuel Arnese}
\address{Department of Industrial Engineering \& Operations Research, Columbia University}
\email{ma4339@columbia.edu}
\thanks{M.A.\ is supported by the Unicredit Marco Fanno scholarship and would like to thank Daniel Lacker for the many and helpful discussions.}
\date{}

\begin{abstract}
We study the rate of propagation of chaos for a McKean--Vlasov equation with conditional expectation terms in the drift. We use a (regularized) Nadaraya--Watson estimator at a particle level to approximate the conditional expectations; we then combine relative entropy methods in the spirit of Jabin and Wang \cite{JWEntropicChaos} with information theoretic inequalities to obtain the result. The nonparametric nature of the problem requires higher regularity for the density of the McKean--Vlasov limit, which we obtain with a bootstrap argument and energy estimates. 
\end{abstract}

\maketitle
\section{introduction}
We study a particle approximation for a system of $m$-coupled Stochastic Differential Equations $X_t=(X_t^i)_{i=1}^m \in \R^{dm}$ that depend on their conditional expectations:
\begin{align}
\label{eq: main SDE}
\begin{split}
    dX^i_t&= \sum_{j=1}^m \E\left[b_j^i(X_t)\,\big \vert\, X_t^j\right]\,dt+V^i(X_t)\,dt+\sqrt 2 \sigma \,dW^i_t;\\
    \text{Law}(X_0)&=\mu_0;\\
    i&=1,\ldots,m,
\end{split}
\end{align}
where each $X_t^i \in \R^d$ and $(W_t^i)_{i=1}^m$ are $m$ independent Brownian motions. Equations of this type have arisen in many different domains. Our main motivation comes from the theory of Wasserstein gradient flows \cite{PLD}, but similar processes have appeared in the study of interacting diffusions on regular trees \cite{local_Equation},\cite{hu2024htheorem},\cite{hu2025case} and for free energy computations in molecular dynamics \cite{BiasingForce}. Related models that we cannot cover in full generality appear in fluid dynamics \cite{JabirModerateInteractions} and in mathematical finance \cite{markovianProjection},\cite{djete2024nonregularmckeanvlasovequationscalibration}.\\
Equation \eqref{eq: main SDE} is a non-linear SDE in the sense of McKean, meaning that its coefficients depend on the law of the solution process. A key feature of \eqref{eq: main SDE} is that the non-linearity is both local and non-local: the drift depends on the density of the solution process both pointwise and through an integral, which is informally clear if we take $\mu_t$ to be the density of $X_t$ and rewrite
\[\E\left[b_j^i(X_t)\,\big\vert\, X^j_t=x^j\right]= \frac{\int b^{i}_j(x^j,x^{-j})\mu_t(x^j,x^{-j})\,dx^{-j}}{\int \mu_t(x^j,x^{-j})\,dx^{-j}},\]
where we abuse notation and write $x=(x^j,x^{-j})$ for some $x\in \R^{dm},x^j \in \R^d,x^{-j}\in \R^{d(m-1)}$. Moreover the coefficients are somewhat singular due to the denominator. These elements complicate the analysis and require an approximation with a moderately interacting particle system \cite{Oelschlger1985}. Let $K_h$ be a mollification kernel with variance $h>0$. We will consider the particle system $\boldsymbol X_t= (X^k_{t})_{k=1}^n$:
\begin{align}
\label{Eq: Particle System}
\begin{split}
    dX^{k,i}_{t}&= \sum_{j=1}^m\frac{\frac1n\sum_{\ell=1}^n b_j^i(X^{k,j}_t,X^{\ell,-j}_{t}) K_h(X^{k,j}_t-X_{t}^{\ell,j})}{\max\left( \frac 1n\sum_{\ell=1}^n K_h(X_{t}^{k,j}-X^{\ell,j}_{t}),\varepsilon\right)}\,dt+V^i(X^{k}_{t})\,dt+\sqrt 2 \sigma \,dW^{k,i}_{t};\\
    \text{Law}(\boldsymbol X_0)&= \mu_0^{\otimes n};\\
    i&=1,\ldots,m;\,\, k=1,\ldots,n.
\end{split}
\end{align}
We approximate the conditional expectation term at a particle level with the celebrated Nadaraya--Watson estimator:
\[\frac{\frac1n\sum_{\ell=1}^n b_j^i(X^{k,j}_t,X^{\ell,-j}_{t}) K_h(X_{t}^{k,j}-X^{\ell,j}_{t})}{\frac 1n\sum_{\ell=1}^n K_h(X_{t}^{k,j}-X^{\ell,j}_{t})}.\]
This is a standard estimator from the theory of nonparametric statistics (see, for example, \cite{Tsybakov2009} or \cite{nonparametric_regression_book}). We further regularize the problem by requiring the denominator to be above a free parameter $\varepsilon$ to make the drift non-singular.  A similar particle method has previously appeared in \cite{BiasingForce} and \cite{BossyConditionalLagriangian}, albeit with a different choice of regularization, and is widely used in practice to implement stochastic local volatility models, see \cite[Chapter 11]{GuyonNonLinearOptionpricing}. Our choice of regularization is taken from the literature on Empirical Bayes \cite{EB1},\cite{EB2} and Score Estimation \cite{WibisonoScoreEstimation}.\\

\subsection{Main result}
Our main result is a convergence rate in total variation on path space of the particle system \eqref{Eq: Particle System} to the McKean--Vlasov limit \eqref{eq: main SDE}. Our main assumption is that the initial law $\mu_0$ and the coefficients $b=(b^i_j)_{i,j\in [m]},V=(V^i)_{i \in [m]}$ are regular enough.

For two probability measures $\alpha,\beta$ on some measurable space, we will denote by 
\[\|\alpha-\beta\|_{\text{TV}}=2 \sup_{S}|\alpha(S)-\beta(S)|\]
the total variation distance between $\alpha$ and $\beta$. Let $\pi[t]$ be the law of the solution of \eqref{Eq: Particle System} on path space up to time $t>0$, and let $\pi^k[t]$ be the marginal of its first $k$ particles. In the same way we denote by $\mu[t]$ the law of \eqref{eq: main SDE} on path space up to time $t$, and by $\mu^{\otimes k}$ the $k$-product of independent copies of $\mu$.

\begin{theorem}
\label{th: main}
Assume that
\begin{enumerate}
    \item the functions $b_j^i,V^i: \R^{dm} \to \R^d$ are bounded with bounded first and second derivatives;
    \item the initial density $\mu_0$ satisfies assumption \ref{assumption: initial regularity} (explained below);
    \item the kernel $K$ satisfies assumption \ref{assumption: Kernel} (explained below).
    \end{enumerate}
    There exists unique weak solutions to \eqref{eq: main SDE} and \eqref{Eq: Particle System}; let $\mu$ be the law of the McKean--Vlasov limit \eqref{eq: main SDE} and let $\pi$ be the law of the particle system \eqref{Eq: Particle System}.
    Then for every $t>0$ and $r<1$, there exists a constant $C$ depending on $t,r,V,b,\mu_0,m,\sigma,K$ such that
    \[\|\mu^{\otimes k}[t]-\pi^k[t]\|_\textup{TV}\leq C \sqrt k \left(\frac{e^{C/(h^d \varepsilon^2)}}{\sqrt {n h^d} \varepsilon}+h+\varepsilon^{r\slash 2}\right).\]
    In particular, choosing $\varepsilon$ and $h$ optimally we obtain for every $r<1$,
    \[\|\pi^k[t]-\mu^{\otimes k}[t]\|_{\text{TV}}\leq C \frac{\sqrt k}{(\log n)^{\frac{r}{dr+4}}}\]
\end{theorem}

Unlike the case of weak interactions where it is possible to prove a convergence rate of order $k\slash n$ with local methods (see \cite{SharpChaosLacker2023}) or $\sqrt{k\slash n}$ with global ones \cite{JWEntropicChaos}, we obtain a logarithmic rate of convergence. This is common for models with moderate interactions (see for instance \cite{Meleard1988} or \cite{ChaintronChaos1}). 
\begin{remark}
    The bound on path space implies a bound on the time marginals by the data processing inequality; indeed for every $s\leq t$,
    \begin{equation*}
        \|\mu^{\otimes k}_s - \pi_s^k\|_{L^1(\R^d)}\leq \|\mu^{\otimes k}[t]-\pi^k[t]\|_\text{TV}.
    \end{equation*}
    where $\mu_s,\pi_s$ are the $s$-time marginals of $\mu,\pi$.
\end{remark}

\subsection{Applications}
Equation \eqref{eq: main SDE} may seem exotic and is certainly not a standard McKean--Vlasov equation. Nonetheless, similar models appear in many applied areas where a particle implementation is of interest.

The authors of \cite{PLD} study the following system of coupled SDEs:
\begin{align}
\label{eq: PLD}
\begin{split}
    dX_t &= \left(\E[\nabla_x c(X_t,Y_t)\mid X_t]\, - \nabla_x c(X_t,Y_t)\right) dt-\varepsilon \nabla U(X_t)\,dt+\sqrt{2\varepsilon} dW_t\\
    dY_t&= \left(\E[\nabla_y c(X_t,Y_t) \mid Y_t ]-\nabla_y c(X_t,Y_t)\right)dt-\varepsilon \nabla V(Y_t)\,dt+\sqrt{2\varepsilon}dB_t.
\end{split}
\end{align}
They argue that the above SDE in dimension $d=1$ is a Wasserstein gradient flow for the free energy
\[\int c(x,y)\,dP(x,y)+\varepsilon \ent(P\,\|\,e^{-U}\otimes e^{-V})\]
(where $\ent$ is relative entropy) constrained to remain in the space of couplings of $e^{-U}$ and $e^{-V}$; informally we can see equation \eqref{eq: PLD} as a Wasserstein gradient flow for the entropic optimal transport problem. They prove that its law converges to the entropic optimal transport coupling in every dimension. 

Our methods adapt seamlessly to the more general equation
\begin{align*}
    \begin{split}
        dX_t^i = \left(\E[\nabla_{x_i} c( X_t)\mid X^i_t] - \nabla_{x_i} c( X_t)\right) dt-\varepsilon \nabla U_i(X^i_t)\,dt+\sqrt{2\varepsilon} dW^i_t;\quad\quad i=1,\ldots m
    \end{split}
\end{align*}
which can be taken to be a Wasserstein gradient flow for the multi marginal entropic Optimal transport problem where the marginals are given by $e^{-U_i}, i=1\ldots m$ (see \cite[Section 1.3.3]{PLD}), in dimension $d=1$.\\

McKean--Vlasov equations with conditional expectations have also appeared in the study of diffusions on $m$-regular infinite graphs: \cite{local_Equation} study a stationary solution of the SDE
\begin{align}
\label{eq: localEquation}
\begin{split}
    dX_t &= -\left((m-1)\E[K'(X_t-Y_t)\mid X_t]\, + K'(X_t-Y_t)\right) dt- U'(X_t)\,dt+\sqrt{2} dW_t\\
    dY_t&= -\left((m-1)\E[K'(Y_t-X_t) \mid Y_t ]+K'(Y_t-X_t)\right)dt-  U'(Y_t)\,dt+\sqrt{2}dB_t.
\end{split}
\end{align}
This SDE has the important property that its stationary distribution corresponds to a $2$-vertex marginal of the stationary distribution of an infinite-dimensional SDE indexed by the vertices of the $m$-regular Cayley Tree $\mathbb T_m$.
 
A related problem is studied in \cite{hu2024htheorem} and \cite{hu2025case}, who consider the long time behavior of the $m$-regular Markovian Local Field Equation (MLFE):
\begin{align}
\label{eq: MLFE}
\begin{split}
    dX^0_t&=-\nabla U(X_t^0)\,dt-\sum_{i=1}^m\nabla W(X_t^0-X_t^i)\,dt+\sqrt{2}dW_t^0\\
    dX^i_t&=-\gamma_t(X_t^i,X_t^0)\,dt+\sqrt{2}dW_t^i; \,\, i=1,\dots,m.
\end{split}
\end{align}
where 
\[\gamma_t(x,y)=\E\left[\nabla U(X_t^1)+\sum_{i=1}^m\nabla W(X_t^1-X_t^i)\,\bigg \vert\,X_t^0=x,X_t^1=y\right].\]
The MLFE can be seen as a Markovian equivalent of a more complicated non-Markov diffusion process, the (Non-Markovian) Local Field Equation, that characterizes neighbourhood-marginal of a system of interacting diffusions on an $m$-regular graph.
This process does not quite fit in our framework due to the reversal of the inputs in $\gamma_t$; nonetheless we expect that our arguments could generalize to this different set up.\\

Our results also apply to the Adaptive Biasing Force Process studied in \cite{BiasingForce}:
\begin{align}
\label{eq:adaptive Biasing Force}
\begin{split}
    dX_t^1 &= -\partial_1 V(X_t)\,dt+\E[\partial_1 V(X_t)\mid X_t^1]\,dt+\,dW_t^1,\\
    dX_t^{-1}&=-\nabla_{-1} V(X_t)\,dt+\,dW_t^{-1},
\end{split}
\end{align}
where we write, for a vector in $\R^d$, $x=(x^1,x^{-1})$ and for a function $V:\R^d \to \R$, $\nabla V=(\partial_1 V,\nabla_{-1} V)$. Equation \eqref{eq:adaptive Biasing Force} is motivated by a computational procedure in the context of molecular dynamics. Working on the Torus, \cite{BiasingForce} provides a logarithmic rate of convergence (on bounded time intervals) for a Nadaraya--Watson based particle system analogous to \eqref{Eq: Particle System}. We extend their convergence result on the whole space and generalize it to a wider class of models; indeed, equation \eqref{eq:adaptive Biasing Force} has the intriguing property that the law of first marginal $X^1_t$ solves the heat equation, a feature that provides a lot of tractability on the Torus and that fails in other models.\\

The behavior of the system at infinity is of critical importance for the above models: the goal of \eqref{eq: PLD}, \eqref{eq:adaptive Biasing Force} and \eqref{eq: MLFE} is to sample from a stationary measure. Unfortunately our result break down for infinite time horizons. This is to be expected at the level of generality of Equation \eqref{eq: main SDE}: some sort of dissipativity is usually required to prove uniform in time propagation of chaos, and we do not assume any. But the difficulty of the problem goes beyond making the right assumptions: the conditional expectations imply that dissipativity assumptions on the coefficient $b$ do not translate immediately to the actual drift. Moreover the moderate interactions pose a challenge, since it is difficult to control the behavior of the particle systems uniformly in $h$ for large $T$. To the best of our knowledge there are no results on uniform in time propagation of chaos for moderate interactions. \\

More complicated Conditional McKean--Vlasov Models have found applications to the study of turbulent flow dynamics. In particular, \cite{BossyJabir2019} study the well posedness of the system
\begin{align*}
    \begin{split}
        dX_t&= b(X_t,Y_t)\,dt+\sigma(X_t)\,dW_t;\\
        dY_t&= \E[\ell(Y_t)\mid X_t]\,dt+\E[\gamma(Y_t)\mid X_t]\,dB_t.
    \end{split}
\end{align*}
Our results cover the simpler case where $\sigma$ and $\gamma$ are constant. This is in general not enough for an accurate simulation of turbulent flows, which is known to be a very hard problem. Another model for turbulent flows studied in \cite{BossyConditionalLagriangian} is 
\begin{align*}
    \begin{split}
        dX_t&= Y_t\,dt;\\
        dY_t&= \E[b(X_t,Y_t)\mid X_t]\,dt+\E[\sigma_t(X_t, Y_t)\mid X_t]\,dB_t.
    \end{split}
\end{align*}
This equation displays a degenerate diffusion coefficient that complicates the regularity estimates needed to close the argument. Our total variation estimates remain essentially correct if we assume $\sigma$ to be constant, but we cannot prove enough regularity of the solution density to check that the constant is finite. Similar models are studied in \cite{Bernardin2010},\cite{Bossy2015},\cite{Bossy2011}. \\

Another application of interest comes from stochastic local volatility models in finance, for example from \cite{markovianProjection}, \cite{djete2024nonregularmckeanvlasovequationscalibration}, and in book form, \cite[Chapter 11]{GuyonNonLinearOptionpricing}.  \cite{markovianProjection} studies the two dimensional SDE system
\begin{align*}
    \begin{split}
        dX_t&= b_1(X_t) \cdot \frac{h(X_t)}{\E[h(X_t)\mid Y_t]}\,dt + \sigma_1(X_t)\cdot \frac{f(X_t)}{\sqrt{\E[f^2(X_t)\mid Y_t]}}\,dW_t;\\
        dY_t&= b_2(Y_t)\,dt+\sigma_2(Y_t)\,dB_t.
    \end{split}
\end{align*}
Our arguments cannot handle the presence of interactions at the denominator of the diffusion coefficient: the total variation estimates do not hold, and extending the regularity result seems challenging. Recent independent work on this issue appeared during the completion of this article: the authors of \cite{SVLFritz} prove quantitative bounds on an Euler approximation for the time marginals of the one-dimensional SDE
\[dX_t= \sigma_t(X_t)\cdot \frac{\xi_t}{\E[\xi^2_t\mid X_t]}dW_t,\]
where $(\xi_t)_{t\geq 0}$ is a given stochastic process. They obtain a weak error rate for their numerical scheme with a precise dependence on the discretization and mollification parameters; their proofs are based on a new half-step scheme. Their error rates with respect to the particle approximation are comparable with ours. 
\subsection{Related Literature}

The problem of propagation of chaos is a classical one, and we refer the interested reader to the nice reviews \cite{ChaintronChaos1} and \cite{ChaintronChaos2} for a comprehensive introduction. We are part of the growing literature on propagation of chaos for equations with moderate interactions. This approach was pioneered (with qualitative results) by \cite{Oelschlger1985}, with first quantitative estimates provided by \cite{Meleard1988} with a logarithmic rate. More recent example are the analysis of propagation of chaos for the viscous Burger's equation carried out in \cite{TomasevicBurgers} and the study of the convergence rate for moderately interacting particle systems driven by an $\alpha$-stable process \cite{JabirModerateInteractions}. The article \cite{Holzinger} is closer to us in approach: they prove a quantitative convergence rate in total variation for a particle approximation of the viscous porous medium equation using entropy methods. Nonetheless, the finer techniques involved in the estimates are quite different: they combine a powerful technical estimate from \cite{OeschlagerFluctuations} with Moser-type inequalities, while we use information-theoretic inequalities. Moreover, we cannot rely on the theory of semilinear PDE to prove the required regularity of the solution density. Let us also mention the impressive follow up paper \cite{HolzingerFluctuations} which proves a fluctuation type result for a moderate interactions model with singular potential. \\

Another class of equations with conditional expectations come from the Cucker-Smale (CS) flocking model (\cite{Cucker-Smale}, \cite{Cucker-Smale2}) with strong alignment introduced by \cite{Motsch2011-Alignment}. These models fall outside of the realm of applicability of our result because of the degeneracy of diffusion coefficient; for recent progress on error rate for an interacting particle system representation of the CS model, see \cite{wang2024errorestimationmeanfieldlimit}.

\subsection{Bias-Variance Trade-off}
A classical method of nonparametric statistics is to analyze separately the bias introduced by the mollification of the empirical measure and the decrease in variance it causes. This approach was also pioneered in the study of equations with moderate interactions by \cite{Oelschlger1985}; more recently, it was used to great effect in \cite{Holzinger} to prove a very tight convergence rate in total variation for a particle approximation for the viscous porous medium equation. We utilize this method as well, and thus introduce an intermediate problem: 
\begin{align}
\label{eq: intermediate problem}
\begin{split}
        dX^i_t &=\sum_{j=1}^m \frac{\int b_j^i(X_t^j,y^{-j})K_h(X_t^j-y^j)\,d\rho_t(y)}{\max\left(\varepsilon,\int K_h(X_t^j-y^j)\,d\rho_t(y)\right)}\,dt+V^i(X_t)\,dt+\sqrt2 \sigma \,dW_t^i;\\
        \text{Law}(X_t)&=\rho_t;\\
        \text{Law}(X_{0})&=\mu_0.\\
        i&=1,\ldots,m.
        \end{split}
\end{align}
Informally, the intermediate problem \eqref{eq: intermediate problem} is the large $n$ limit of a marginal of the particle system \eqref{Eq: Particle System} when we keep the mollification parameter $h$ fixed. The intermediate problem is a McKean--Vlasov equation, but unlike \eqref{eq: main SDE}, it is completely non-local and thus more tractable. We can then use the triangle inequality to estimate
\[\|\mu^{\otimes k}[T]-\pi^k[T]\|_\text{TV}\leq \underbrace{\|\mu^{\otimes k}[T]-\rho^{\otimes k}[T]\|_\text{TV}}_{\text{Bias}}+\underbrace{\|\rho^{\otimes k}[T]-\pi^k[T]\|_\text{TV}}_{\text{Variance}}\]

The analysis of the two terms is rather different. The variance term requires a propagation of chaos result for equation \eqref{eq: intermediate problem}. While the drift is somewhat non standard, the intermediate problem features weak interactions and is thus easier to handle, and we can adapt the methods from \cite{JWEntropicChaos} to obtain a rate of convergence. In our particular setting there is no need for a large deviation principle, but we instead need to provide quantitative convergences rates for the (mollified) Nadaraya--Watson estimator on an exponential scale. \\

We handle the bias term by studying the relative entropy $\ent(\mu_t\,\|\,\rho_t)$. The starting point is to recognize that the drift of the intermediate problem remains a conditional expectation (ignoring the regularization for the moment): indeed we can rewrite
\begin{equation*}
\frac{\int b^i_j(x^j,y^{-j})K_h(x^j-z)\,d\rho_t(z,y^{-j})}{ \int K_h(x^j-z)\,d\rho_t(z,y^{-j})}= \E_{\rho}\left[b^i_j(X^j_t+ Z,X^{-j}_t)\mid X^j_t+ Z=x^j\right],    
\end{equation*}
where $Z$ is a random variable with distribution $K_h$ that is independent of $(X_t,Y_t)$. The problem then becomes comparing two conditional expectations. We do so by using information theoretic inequalities to relate the conditional distributions and then cancel out the denominator; this reduces the problem to estimating  $\ent(\mu_t\,\|\,\mu_t\star_i K_h)$ for every $i$ (where $\star_i$ denotes convolution on the $i$-th space coordinate), which requires a good understanding of the regularity properties of $\mu_t$.

\subsection{Regularity}
It is well known in statistics that the quality of nonparametric estimation depends on the regularity of the object to be estimated. In the same way, we will require some regularity for the density of \eqref{eq: main SDE} and the conditional expectations terms. This difficult problem is common in the literature of quantitative propagation of chaos for equations with moderate interactions, but better studied models (such as the Viscous Burger's equation studied in \cite{TomasevicBurgers} or the Viscous Porous medium analyzed by \cite{Holzinger}) can rely on the theory of semilinear parabolic PDEs from, for example, \cite{ladyženskaja1988linear}. Regularity for equations that involve conditional expectations is instead not well understood, a source of difficulty encountered in previous works in the area such as \cite{PLD} and \cite{hu2024htheorem}.\\

For the specific case of the Nadaraya--Watson estimator, references such as \cite{StoneNPR} and (in book form) \cite{nonparametric_regression_book} give convergence rates for H\"older continuous conditional expectation functions; further results are available under the assumption that the conditional expectation function has the appropriate Sobolev regularity (see \cite{Tsybakov2009}). In general, control over the second derivatives is required to obtain good properties. Waving away differentiability issues for now, an informal computation tells us that 
\begin{align*}
    \partial^2_\ell \E[b_i^j(X_t)\mid X^j_t=x]&=\E[b_i^j(X_t)\partial^2_{\ell}\log \mu_t(X_t) \mid X^j_t=x]\\    &+\E[\partial^2_{\ell} b_i^j(X_t)\mid X^j_t=x]+\text{Terms with lower order derivatives}.
\end{align*}
We will thus need to control the logarithmic Hessian of $\mu_t$ under conditional expectations. 

The problem of estimating the logarithmic Hessian has appeared in previous works on propagation of chaos, for instance \cite{WangScoreEstimates}, \cite{UiTLSI} and \cite{modulatedenergy}. The approaches in \cite{wang2024errorestimationmeanfieldlimit} and \cite{modulatedenergy} are based on the maximum principle; this method seems out of reach because working with conditional expectations terms at a pointwise level is complicated (for instance, conditional expectations do not preserve local bounds).  On the other hand, \cite{UiTLSI} utilizes a stochastic control representation of the log density  which requires regularity of the coefficients that is not available in our setting. 
We must instead develop novel estimates for the moments of the weak derivative $\E[|\partial^2_\ell \E[b^i_j(X_t)\mid X_t]|^p]$ (and indeed prove that the weak derivatives exist in the first place).\color{black} \\

Our approach is based on a bootstrap argument. The $i$-th (weak) derivatives of the drift depend on the $i$-th (weak) derivative of $\log\mu_t$, which in turn depends on the $(i-1)$-th derivative of the drift. This allows us to bootstrap our way from control of the zero-th order terms to third derivatives, the most that we require.\\

Finally, at a technical level, we prove the required regularity by energy estimates in the spirit of \cite[Chapter 7]{bogachev2022fokker} at higher orders and a careful mollification argument. Similar estimates have previously appeared in \cite{Chaintron2024}, in a different context (with well behaved drift but non-trivial diffusion coefficient), but were obtained by a time reversal argument.

\subsection{The assumptions \ref{assumption: initial regularity} and \ref{assumption: Kernel}}
We make the following assumptions on the kernel $K$;
\begin{assumption}{K}
\label{assumption: Kernel}
Let $K:\R^d \to \R_+$ be a probability density that satisfies
\begin{align*}
    \int_{\R^d} K(z)z\,dz=0\,;\, \int_{\R^d} K(z) e^{|z|^2/4}\,dz<\infty\,;\,K(z)=K(-z)\,;\,\|K\|_\infty<\infty.
\end{align*}
\end{assumption}
These assumptions are satisfied by natural choices of kernels. We denote
\[K_h(z)=\frac{1}{\sqrt{h^d}}K\left(\frac{z}{\sqrt h}\right),\]
so that $\int |x|^2\,dK_h \leq Ch$ for some $C$ that depends on $K$.\\

We will repeatedly abuse notation and write $x=(x^j,x^{-j})$ for $j \in [m]$, where $x\in \R^{dm},x^j \in \R^d$ and $x^{-j}\in \R^{d(m-1)}$.

For probability densities $\alpha,\beta \in \mathcal{P}(\R^{dm})$ and $\gamma \in \mathcal{P}(\R^d)$, we denote their convolution as 
\begin{align*}
    \alpha\star\beta(x)= \int \alpha(x-z)\beta(z)\,dz\,\,;\,\,\alpha\star_j \gamma = \int \alpha(x^j-z,x^{-j})\gamma(z)\,dz.
\end{align*}
We recall here the definition of the $p$-divergence as
\begin{align*}
    \D_p(\alpha\,\|\,\beta)=&\begin{cases}\int \left(\frac{d\alpha}{d\beta}\right)^p\,d\beta \text{ if }\alpha\ll\beta;\\
    \infty \text{ otherwise},
    \end{cases}
\end{align*}
and the related Chi square divergence as $\chi^2(\alpha\,\|\,\beta)=D_2(\alpha\,\|\,\beta)-1$.

As is common in the literature on propagation of chaos for moderate interactions, we require substantial regularity from the starting distribution $\mu_0$. We will often abuse notation and identify a measure with its density.

\begin{assumption}{R}
\label{assumption: initial regularity}
The law $\mu_0$ of $X_0$ admits a density that satisfies the following:
\begin{enumerate}
    \item For each $i \in [m]$, 
    \[\int_0^1\int_0^1\int_{\R^d}\int_{\R^d} D_4(\mu_0 \star_i \delta_{\sqrt h (s u z_1+  z_2)}\,\|\,\mu_0)\,dK(z_1)\,dK(z_2)\,ds\,du<\infty.\]
    \item For every $p\geq 1$, $\E[|X_0|^p]<\infty$.
    \item The initial density $\mu_0$ is in $H^2(\R^d)$  and 
    \[\int_{\R^{dm}}|\log \mu_0(x)|+\left|\Delta \mu_0(X_0)/\mu_0(x)\right|^2+|\nabla \log \mu_0(x)|^4 \,d\mu_0(x)<\infty.\]
    \item The density $\mu_0$ is locally bounded away from $0$ and $\mu_0\in L^\infty(\R^d)$.
\end{enumerate}
\end{assumption}
Assumption \ref{assumption: initial regularity}(1) is technical. It requires the initial distribution $\mu_0$ to have small enough $4$-Divergence in expectation with respect to a $K_h$-perturbation. This condition arises from the use of mollification, and essentially tells us that the mollification does not change the initial distribution too much. While this condition is certainly non-standard, we will provide common examples of distributions that satisfy it; it is essentially a sub-Gaussian requirement for the logarithmic gradients of $\mu_0$.\bigskip\\
Assumption \ref{assumption: initial regularity}.2 is related to the $\varepsilon$-approximation of the denominator in the particle system \eqref{Eq: Particle System}. This approximation is exact on a set with a large amount of mass (that is, the set where $\mu_t^i$ is greater than $\varepsilon$), and is thus quite effective. Having all polynomial moments be finite implies a strong decay of $\mu_0$ (and thus of $\mu_t$), and assures us that the approximation is good on a large portion of space.\bigskip\\
Finally, Assumption \ref{assumption: initial regularity}.3 is required in order for the Nadaraya-Watson estimator to approximate the true conditional expectation well. Classical results from nonparametric statistics show that optimal convergence rates of the Nadaraya-Watson estimator require the conditional expectation function to have well behaved derivatives (see, for example, \cite{StoneNPR} or \cite[Chapter 5]{nonparametric_regression_book} for a book form introduction). Assumption R.3 guarantees that $\mu_0$ has enough regularity to obtain a good convergence rate; a major hurdle of the analysis is proving that the regularity propagates in time. \bigskip \\
Assumption \ref{assumption: initial regularity} is surely strong, but it is satisfied by many natural distributions since $\sqrt h$ tends to be small.
\begin{example}
Let $\mu_0$ be sub-exponential, meaning that for some $\delta>0$ it holds
\begin{equation*}
    \int_{\R^{dm}} e^{\delta |x|}\mu_0(x)\,dx<\infty.
\end{equation*}
Assume $\mu_0$ is twice differentiable with bounded score $\nabla \log \mu_0$ and $|\Delta \log \mu_0(x)|\leq C(1+|x|^p)$ for some $p\geq 1$, and assume that $0<\mu_0<C$ for some $C>0$. Conditions \ref{assumption: initial regularity}(2-4) clearly hold. 
For an arbitrary $z$ we can compute
\begin{align*}
    \log \mu_0(x+z)-\log \mu_0(x)=\int_0^1 \nabla \log \mu_0(x+sz)\cdot z\,ds,
\end{align*}
which implies
\[\frac{\mu_0(x+z)}{\mu_0(x)}= \exp\left\{\int_0^1 \nabla \log \mu_0(x+sz)\cdot z\,ds\right\}.\]
We thus obtain the bound
\begin{align*}
    \D_4(\mu_0\star \delta_{z\sqrt h}\,\|\,\mu_0)&=\int_{\R^{dm}} \exp\left\{4\sqrt{h}\int_0^1 \nabla \log \mu_0(x+\sqrt hzs)\cdot z\,ds\right\}\,d\mu_0(x)\\
    &\leq e^{4\sqrt{h}\|\nabla \log \mu_0\|_\infty |z|}.
\end{align*}
And thus assumption \ref{assumption: initial regularity}(1) is satisfied if $K$ satisfies Assumption \ref{assumption: Kernel}. 
\end{example}
\begin{example}
    Let $\mu_0$ be subgaussian, meaning that for some $\delta>0$ it holds 
    \[\int_{\R^{dm}}e^{\delta |x|^2}\mu(x)\,dx<\infty.\]
    Assume $\mu_0$ is twice differentiable, $0<\mu_0<C$ and it satisfies $-LI\leq\nabla^2 \log \mu_0\leq LI$ for some $L,C>0$. Conditions \ref{assumption: initial regularity}(2-4) are satisfied. By \cite[Corollary 5.3]{ChewiHarnack}, $\mu_0$ satisfies the reverse transport inequality
    \begin{align*}
        D_4(\mu_0 \star \delta_{z}\,\|\,\mu_0)\leq e^{2L|z|^2}.
    \end{align*}
    Taking $z=\sqrt h (s u z_1+ z_2)$ with $s,u\leq 1$, we find that 
    \begin{align*}
\D_4&(\mu_0\star \delta_{\sqrt h (s u z_1+ z_2)}\,\|\,\mu_0)\leq e^{4Lh|z_1|^2+4Lh|z_2|^2}.
    \end{align*}
    Hence Assumption \ref{assumption: initial regularity}(1) is satisfied if $8Lh<1$, assuming Assumption \ref{assumption: Kernel} (or that $K$ is compactly supported). 
\end{example}

Finally we discuss in more detail the assumptions of Theorem \ref{th: main}. The assumption of boundedness of $V$ can likely be relaxed at the cost of a more complicated proof for the regularity properties of $\mu$. On the other hand, relaxing boundedness of $b$ seems difficult, as we make generous use of information theoretic inequalities that do not translate as well to unbounded conditional expectation terms. Moreover, we note that even existence of a solution for the McKean--Vlasov problem \eqref{eq: main SDE} is difficult without bounded $b$. 
\begin{remark}
    Assuming only \ref{assumption: initial regularity}(2), it can be showed that $\mu_t$ will satisfy assumption \ref{assumption: initial regularity}(3) and \ref{assumption: initial regularity}(4) for almost every $t>0$, regardless of $\mu_0$. Indeed \ref{assumption: initial regularity}(3) is a  consequence of the analysis in Section \ref{sect: regularity estimates}, while \ref{assumption: initial regularity}(4) comes from the usual regularity properties of Fokker--Planck equations. 
\end{remark}

\subsection{Outline}
The rest of the article is organized as follows. In Section \ref{section: particles to intermediate} we analyze the variance term $\|\pi^k-\rho^{\otimes k}\|_\text{TV}$ and thus the convergence rate of the particle system \eqref{Eq: Particle System} to the intermediate problem \eqref{eq: intermediate problem}. Section \ref{section: intermediate to McKean--Vlasov} studies the bias term $\|\mu^{\otimes k}-\rho^{\otimes k}\|_\text{TV}$ and quantifies the distance between the intermediate problem \eqref{eq: intermediate problem} and the McKean--Vlasov limit \eqref{eq: main SDE} as a function of the mollification parameter $h$. Section \ref{sect: regularity estimates} contains the regularity estimates that are required to close the proofs in Section \ref{section: intermediate to McKean--Vlasov}. Appendix \ref{appendix: well posedness} contains well-posedness result for equation \eqref{eq: main SDE}, and we relegate some routine but necessary calculations to Appendix \ref{Appendix: mollification estimates}.
\section{Particle System to Intermediate Problem}
\label{section: particles to intermediate}
In this section we study the convergence rate of the particle system \eqref{Eq: Particle System} to the intermediate problem \eqref{eq: intermediate problem} as a function of the free parameters $h,\varepsilon$ and the number of particles $n$. 
For brevity, we will define $\hat b^i_j: \R^d\times \mathcal{P}(\R^{dm})\to \R^d$ as
\[\hat{b}^i_j (x^j,\nu)= \frac{\int b^i_j(x^j,y^{-j})K_h(x^j-y^j)\,d\nu(y)}{\max\left(\varepsilon, \int K_h(x^j-y^j)\,d\nu(y)\right)},\]
Finally, for $\boldsymbol x=(x_i)_{i=1}^n \in \R^{n \times dm }$, we denote the empirical measure of $\boldsymbol x$ as
\[L_n(\boldsymbol x)=\frac1n \sum_{i=1}^n \delta_{x_{i}}\in \mathcal{P}(\R^{dm}).\] 

This allows us to rewrite equations \eqref{Eq: Particle System} and \eqref{eq: intermediate problem} as
\begin{align*}
    \begin{split}
        dX^{k,i}_{t} &= \sum_{j=1}^m\hat b^i_j (X^{k,j}_{t}, L_n(\boldsymbol X_t))\,dt + V^i(X_{t}^k)\,dt+\sqrt2 \sigma \,dW_{t}^{k,i}\\
        \text{Law}(\boldsymbol X_0)&= \mu_0^{\otimes n};\\
    i&=1,\ldots,m; \,\,k=1,\ldots,n.
\end{split}
\end{align*}
and
\begin{align*}
    \begin{split}
        dX^i_t &= \sum_{j=1}^m\hat b^i_j (X_t, \rho_t)\,dt + V^i(X^i_t)\,dt+\sqrt2 \sigma \,dW^i_t\\
        \text{Law}(X_0)&=\mu_0;\\
    \text{Law}(X_t)&=\rho_t;\\
    i&=1,\ldots,m.
\end{split}
\end{align*} 
\begin{lemma}
\label{lm: well posedness}
Assuming that $b,V$ are bounded and that $K$ satisfies assumption \ref{assumption: Kernel}, equations \eqref{Eq: Particle System} and \eqref{eq: intermediate problem} have unique weak solutions for every $\mu_0 \in \mathcal{P}(\R^{dm})$.
\end{lemma}
\begin{proof}
    Take $\zeta,\xi \in \mathcal{P}(\R^{dm})$. Applying the triangle inequality and using the boundedness of $b^i_j$ and of the denominator we get
    \begin{align}
    \label{eq: bound on difference of regularized estimator}
    \begin{split}
        |\hat b_j^i(x^j,\zeta)-\hat b_j^i(x^j,\xi)|&\leq \frac{1}{\varepsilon}\left|\int_{\R^{dm}}b_j^i(x^j,y^{-j})K_h(x^j-z)d(\zeta(z,y^{-j})-\xi(z,y^{-j}))\right|\\
        &+\frac{\|b^i_j\|_\infty}{\varepsilon} \left|\int_{\R^{dm}} K_h(x^j-z)d(\zeta(z,y^{-j})-\xi(z,y^{-j}))\right|.
    \end{split}
    \end{align}
    Since $K_h$ and $b^i_j$ are bounded, we obtain by the definition of total variation
    \begin{align*}
        |\hat b_j^i(x^j,\zeta)-\hat b_j^i(x^j,\xi)|&\leq \frac{(\|b^i_jK_h\|_{\infty}+\|K_h\|_\infty)}{\varepsilon}\|\zeta-\xi\|_{\text{TV}}.
    \end{align*}
    Now we can use \cite[Theorem 2.4]{StrongPropagationOfChaos} to prove that the problem \eqref{eq: intermediate problem} has a unique weak solution for each $\mu_0 \in \mathcal{P}(\R^{md})$. Weak well-posedness for the particle system \eqref{Eq: Particle System} follows by boundedness and a standard Girsanov transformation. 
\end{proof}

We can now check a subgaussian bound on $\hat b^i_j$ to obtain convergence of the particle system to the regularized problem. We use ideas borrowed from \cite{JWEntropicChaos}, and in particular we work at the level of relative entropy. Recall that the relative entropy between two probability measures $\alpha,\beta \in \mathcal P(\R^{dm})$ is defined as
\begin{align*}
    \ent(\alpha\,\|\,\beta)=\begin{cases}
        \int \log\frac{d\alpha}{d\beta}d\alpha \text{ if }\alpha \ll \beta;\\
        \infty \text{ otherwise}.
    \end{cases}
\end{align*}
We will repeatedly use some well known inequalities for information divergences. In particular we will use the Donsker--Varadhan variational representation of relative entropy, which states that for every bounded measurable $\varphi:\R^d \to \R$ we have

\[\int \varphi(x) \,d\alpha(x) \leq \ent(\alpha\,\|\,\beta)+\log \int e^{\varphi(x) }\,d\beta(x).\]

\begin{proposition}[Variance Bound]
\label{pr: particles to intermediate}
Under the assumptions of Theorem \ref{th: main}, there exists a constant $C=C(T,b,\mu_0,K,\sigma,m)$ such that 
\label{lm:subgaussian concentration}
\[\ent(\pi^k[T]\,\|\,\rho^{\otimes k}[T]) \leq \frac{Ce^{\frac{C}{h^d \varepsilon^2}}}{\varepsilon^2h^d}\cdot \frac{k}{n}.\]
\end{proposition}
    \begin{proof}
    Let $\boldsymbol X_t=(X_{t}^1,\ldots,X_{t}^n)$, with $X_{t}^k \in \R^{dm}$ for each $k$ be the coordinate process on path space on $\R^{dm}$; we will denote by $\E_\pi$ expectations taken under the path law $\pi$, and similarly for $\rho^{\otimes n}$. A by now standard entropy estimates (for instance, as done in \cite[Section 3]{StrongPropagationOfChaos}), combined with exchangeability and the triangle inequality yields
    \begin{align}
    \label{eq: first bound on entropy particle system }
    \begin{split}
        \ent (\pi[T]\,\|\,\rho^{\otimes n}[T])&=\frac{n}{4\sigma^2}\int_0^T\sum_{i=1}^m \E_{\pi}\bigg[\Big|\sum_{j=1}^m\hat b^i_j(X_t^{1,j},L(\boldsymbol{X}_t))-\hat b^i_j(X_t^{1,j},\rho_t)\Big|^2\bigg]\,dt\\
        &\leq \frac{nm}{4\sigma^2}\sum_{i=1}^m\sum_{j=1}^m \int_0^T \E_{\pi}\left[\left|\hat b^i_j(X_t^{1,j},L(\boldsymbol{X}_t))-\hat b^i_j(X_t^{1,j},\rho_t)\right|^2\right]\,dt.
    \end{split}
    \end{align}
     Use the variational representation of Entropy to obtain for some $c>0$, 
    \begin{align}
    \label{eq: particles entropy}
    \begin{split}
        \ent (\pi[T]\,\|\,\rho^{\otimes n}[T])&\leq \frac{cm^3}{4\sigma^2} \int_0^T\ent(\pi_t\,\|\,\rho_t^{\otimes n})\,dt\\
        &+\frac{mc}{4\sigma^2}\sum_{i,j}\int_0^T\log \E_{\rho_t^{\otimes n}}\left [\exp\left\{\frac nc\left|\hat b^i_j(X_t^{1,j},L(\boldsymbol{X}_t))-\hat b^i_j(X_t^{1,j},\rho_t)\right|^2\right\}\right]\,dt.
    \end{split}
    \end{align}
    We study a fixed $\hat b^i_j$ term, the others being the same, for a fixed time $t$.
    To control the exponential integral, we can perform a Taylor expansion and study the moments of $|\hat b^i_j(X_t^{1,j},L_n(\boldsymbol X))-\hat b^i_j(X_t^{1,j},\rho_t)|^{2p}$ for every $p\geq 1$.
    By \eqref{eq: bound on difference of regularized estimator}, we only need to check concentration of 
    \[\frac 1n\sum_{s}b^i_j(X^{1,j}_t,X^{s,-j}_t)K_h(X_t^{1,j}-X_t^{s,j}) \text{ and } \frac 1n\sum_{s}K_h(X^{1,j}_t-X^{s,j}_t)\] 
    around their large $n$ limits; their analysis will be similar. 
    Use a uniform bound and check
    \begin{align}
    \label{eq: concentration 1}
    \begin{split}
    \frac{1}{n^{2p}}&\Big|b^i_j(X^{1,j}_t,X^{s,-j}_t)K_h(0)-\int_{\R^{dm}} b^i_j(X^{s,j}_t,y^{-j})K_h(X^{1,j}_t-y^j)\,d\rho_t(y)\Big|^{2p}\leq \frac{\|b^i_j K\|_\infty^{2p}}{n^{2p}}\left(1+\frac{1}{h^{dp}}\right)\\
    &\leq \frac{p!}{n^p}\cdot \Big(\frac{\|b^i_j K\|_\infty^{2}}{h^{d}}\Big)^p.
    \end{split}
    \end{align}
    On the other hand, since $X^1$ is independent from $(X^s)_{s=2}^n$ under $\rho^{\otimes n}$, we can condition and use the freezing lemma to see
      \begin{align}
    \label{eq: concentration 2}
    \begin{split}
        &\E_{\rho^{\otimes n}}\bigg[\bigg|\frac 1n \sum_{s=2}^n \bigg(b^i_j(X_t^{1,j},X^{s,-j}_t) K_h(X^{1,j}_t-X^{s,j}_t)-\int_{\R^{dm}} b^i_j(X^{1,j}_t,y^{-j})K_h(X^{1,j}_t-y^j) \,d\rho_t(y)\,\bigg)\bigg|^{2p}\bigg]\\
        &=\int_{\R^d} \E_{\rho^{\otimes n}}\bigg[\bigg|\frac 1n \sum_{s=2}^n \bigg(b^i_j(z,X_t^{s,-j}) K_h(z-X^{s,j}_t)-\int_{\R^{dm}} b_i^j(z,y^{-j})K_h(z-y^j)\,d\rho_t(y)\bigg)\bigg|^{2p}\bigg \vert X^{1,j}_t=z\bigg]\,d\rho^j_{t}(z)\\
        &=\int_{\R^d} \E_{\rho^{\otimes n}}\bigg[\bigg|\frac 1n \sum_{s=2}^n\bigg( b^i_j (z,X_t^{s,-j}) K_h(z-X^{s,j}_t)-\int_{\R^{dm}} b^i_j(z,y^{-j})K_h(z-y^j) \,d\rho_t(y)\bigg)\bigg|^{2p}\bigg]\,d\rho_t^j(z)\\
        &\leq \int_{\R^d} \frac{p!(n-1)^p}{n^{2p}}\|b^i_j(z,\cdot)K_h(z-\cdot) \|^{2p}_\infty\,d\rho_t^j(z) \leq \frac{p!}{n^p}\cdot \left(\frac{\|b^i_jK\|_\infty^{2}}{h^{d}}\right)^p.
        \end{split}
    \end{align}
    The first inequality follows since a bounded random variable $Z$ is subgaussian with parameter $\|Z\|^2_\infty$ and the mean of $n$, $\sigma^2$-subgaussian independent random variables is subgaussian with parameter $\sigma^2 \slash n$. Plugging \eqref{eq: concentration 1} and \eqref{eq: concentration 2} into \eqref{eq: bound on difference of regularized estimator} (with the same reasoning for the $\sum_k K_h(X_t^{1,j}-X_t^{k,j})$ term), we conclude that for some $C$ that depends on $b$ and $K$,
    \[\E_{\rho_t^{\otimes n}}\left[|\hat b^i_j(X_t^{1,j},L(\boldsymbol X))-\hat b^i_j(X_t^{1,j},\rho_t)|^{2p}\right]^{\frac{1}{2p}}\leq   \frac{C}{\varepsilon \sqrt{n h^{d}}} (p!)^{1/2p}.\]
    Thus if we take $c=2  C^2 / \varepsilon^2 h^{d}$ we get
    \begin{align*}
        c \log \E_{\rho_t^{\otimes n}}&\left [\exp\left\{\frac nc\left|\hat b^i_j(X_t^{1,j},L(\boldsymbol{X}_t)-\hat b^i_j(X_t^{1,j},\rho_t)\right|^2\right\}\right]\\
        &=c \log\left( \sum_{p=0}^\infty \frac{n^p}{c^p p!}\E\left[\left|\hat b^i_j(X_t^{1,j},L(\boldsymbol{X}_t)-\hat b^i_j(X_t^{1,j},\rho_t)\right|^{2p}\right]\right)\\
        &\leq 2  C^2 / \varepsilon^2 h^{d}\log 2.
    \end{align*}
    Plugging this bound into equations \eqref{eq: particles entropy} and using Gronwall's inequality yields for a constant $C$ independent of $n,h,\varepsilon$
    \begin{align*}
        \ent(\pi[T]\,\|\,\rho^{\otimes n}[T])\leq \frac{C}{\varepsilon^2 h^d}e^{\frac{C}{\varepsilon^2 h^{d}}}.
    \end{align*}
    Since $\rho_t^{\otimes n}$ is a product measure and $\pi$ is exchangeable, we can use the well known subadditivity of relative entropy (see for instance Lemma A.4 in \cite{Holzinger}) to see
    \begin{align*}
        \ent(\pi^k[T]\,\|\,\rho^{\otimes k}[T])&\leq \frac kn \ent(\pi[T]\,\|\,\rho^{\otimes n}[T])\\
        &\leq \frac kn \frac{C}{\varepsilon^2 h^d}e^{\frac{C}{\varepsilon^2 h^{d}}}
    \end{align*}
    which completes the proof.
    \end{proof}
We remark that this part of the analysis is likely suboptimal, despite the fact that a logarithmic rate is typical in the moderate interactions literature. To the best of our knowledge, the only papers with moderate interactions (where the kernel $K$ approaches a Dirac's delta) that can prove a polynomial rate are \cite{Holzinger} and \cite{TomasevicBurgers}, for the better understood cases of the viscous porous medium and viscous Burger's equation. 
\section{Regularized Process to McKean--Vlasov Equation}
\label{section: intermediate to McKean--Vlasov}
In this section we study the rate of convergence in relative entropy of the regularized SDE \eqref{eq: intermediate problem} to the McKean--Vlasov limit \eqref{eq: main SDE}. Well posedness of \eqref{eq: main SDE} easily follows by other results in the literature, but to be self contained we provide a proof in Appendix \ref{section: well posedness}.
\begin{lemma}
\label{lm: well posedness MKV}
    Under the assumptions of Theorem \ref{th: main}, Equation \eqref{eq: main SDE} has a unique in law weak solution with law $\mu$.
\end{lemma}
We begin with a technical remark about the representation of the conditional expectation functions.
\begin{remark}
    Since $\rho$ and $\mu$ are the law of SDEs with bounded drift, it follows by \cite[Corollary 6.4.3]{bogachev2022fokker} that they admit locally H\"older densities. By classical results we conclude that for every $\varphi \in L^\infty(\R^{dm})$,
    \[\E_\mu[\varphi(X_t)\mid X_t^j=x^j]=\frac{\int \varphi(x^j,x^{-j}) \mu_t(x^j,x^{-j})\,dx^{-j}}{\int \mu_t(x^j,x^{-j})\,dx^{-j}};\]
    the same applies to $\rho$.
\end{remark}
    
In what follows, the relative entropy $\ent(\mu_t\,\|\, \mu_t \star_j K_h)$ between the law of the McKean--Vlasov limit \eqref{eq: main SDE} at time $t$ and itself mollified in the $j$-th coordinate (for an arbitrary $j$) will be a key object. This quantity is crucial because it measures how well the mollified conditional expectation that appears in the drift of \eqref{eq: intermediate problem} approximates the actual conditional expectation. Its behavior depends on the regularity properties of $\mu$ that will be the focus of Section \ref{sect: regularity estimates}. We control the relative entropy with the following Lemma.
    \begin{lemma}
    \label{lm: bound on time integral of relative entropy}
        There exists a constant $C<\infty$ such that 
        \[\int_0^T \ent(\mu_t\,\|\,\mu_t \star_j K_h)\,dt\leq C h^2.\]
        The constant $C$ depends on $T,b,V,\mu_0,d,\sigma,m$ and $K$.
    \end{lemma}
    We postpone the proof of Lemma \ref{lm: bound on time integral of relative entropy} to the end of the Section \ref{sect: regularity estimates}.
    For a probability density $\alpha \in \mathcal{P}(\R^{dm})$ we will denote by $\alpha^j\in \mathcal{P}(\R^d)$ its $j$-th marginal and by $\alpha^{x^j}$ its conditional density:
\begin{align*}
    \alpha^j(x^j)=\int\alpha^i_j(x^j,y^{-j})\,dy^{-j}\,\,;\,\,
    \alpha^{x^j}(y^{-j})=\frac{\alpha(x^j,y^{-j})}{\alpha^j(x^j)}.
\end{align*}

We recall some information-theoretic inequalities that will be the backbone of our approach.
The first is the chain rule for relative entropy, which states that for probability measures $\alpha,\beta \in \mathcal{P}(\R^{dm})$,
\[\ent(\alpha\,\|\,\beta) = \ent(\alpha^j\,\|\,\beta^j)+\int \ent(\alpha^{x^j}\,\|\,\beta^{x^j})\,d\alpha^j(x^j).\]

Another key property is the data processing inequality: given probability measures $\alpha,\beta,\gamma$ we have
\[\ent(\alpha \star \gamma \,\|\,\beta \star \gamma)\leq \ent(\alpha\,\|\,\beta).\] 
We remark that the data processing inequality holds for $\D_p$ and $\chi^2$ as well. Lastly we recall Pinsker's inequality; for two probability measures $\alpha,\beta$ on some Polish space it holds 
\[\|\alpha-\beta\|^2_{\text{TV}}\leq 2\ent(\alpha\,\|\,\beta).\]

\begin{proposition}
\label{pr: regularized process to MKV}
Under the assumptions of Theorem \ref{th: main}, for each $p<1$ we can find a constant $C$ depending on $p,T,\sigma,\mu_0,b,V,K,m$ such that
     \begin{equation*}
     \label{eq: regularized process to MKV}
         \ent(\mu[T]\,\|\,\rho[T])\leq C(h^2+\varepsilon^{p})
     \end{equation*}
\end{proposition}
\begin{proof}
     We analyze entropy on path space.
    An entropy calculation (see for example \cite[Section 3]{StrongPropagationOfChaos}) and the triangle inequality bring us to the following expression for entropy:
    \begin{align*}
    \begin{split}
    &\ent(\mu[T] \,\|\,\rho[T])\\
    &=\frac{1}{4\sigma^2}\sum_{i=1}^m\int_0^T \int_{\R^{m d}} \bigg|\sum_{j=1}^m\int_{\R^{d (m-1)}} b^{i}_j (x^j,y^{-j})\left(\frac{\rho_t\star_j K_h(x^{j},y^{-j})}{ \rho^{j}_t\star K_h (x^j) \vee \varepsilon} -\frac{\mu_t(x^j,y^{-j})}{\mu_t^j(x^{j})}\right)\,dy^{-j}\bigg \vert^2\,d\mu_t(x)\,dt\\
    &\leq \frac{m}{4\sigma^2}\sum_{i=1}^m \sum_{j=1}^m\int_0^T \int_{\R^{d}} \bigg|\int_{\R^{d (m-1)}} b^{i}_j (x^j,y^{-j})\left(\frac{\rho_t\star_j K_h(x^{j},y^{-j})}{ \rho^{j}_t\star K_h (x^j) \vee \varepsilon} -\frac{\mu_t(x^j,y^{-j})}{\mu_t^j(x^{j})}\right)\,dy^{-j}\bigg \vert^2\,d\mu^j_t(x^j)\,dt.
    \end{split}
    \end{align*}
    Fix $j \in [m]$; we study a generic $b_j^i$ term, and the analysis of the other ones will follow naturally. From now on we will use the shorthand notation:
    \begin{align*}
        \tilde \rho_t= \rho_t \star_j K_h\,\,;\,\,\tilde\rho_t^j= \rho_t^j \star K_h\,\,;\,\,\tilde \rho_t^{x^j}(x^{-j})=\frac{\tilde\rho_t(x^j,x^{-j})}{\tilde\rho_t^j(x^j)}.
    \end{align*}
    We define $\tilde \mu_t, \tilde \mu^j_t$ and $\tilde \mu_t^{x^j}$ analogously. Use the triangle inequality to see
    \begin{align}
    \begin{split}
    \label{eq: bound on a}
        \int_0^T &\int_{\R^d} \left|\int_{\R^{d(m-1)}}  b_j^i(x^j,y^{-j})\left(\frac{\tilde\rho_t(x^j,y^{-j})}{\tilde \rho^j_t (x^j) \vee \varepsilon} -\mu_t^{x^j}(y^{-j})\right)\,dy^{-j}\right|^2\,d\mu_t^j(x^j)\,dt\\
        &\leq 3\int_0^T\int_{\R^d} \left\vert \int_{\R^{d(m-1)}}  b_j^i(x^j,y^{-j})( \mu_t^{x^j}-\tilde \mu_t^{x^j})(y^{-j})\,dy^{-j}\right\vert^2\,d\mu_t^j(x^j)\,dt\\
        &+3\int_0^T \int_{\R^d}\left\vert \int_{\R^{d(m-1)}} b_j^i(x^j,y^{-j})\left( \tilde \mu_t^{x^j}-\tilde\rho_t^{x^j}\right)(y^{-j})\,dy^{-j} \right\vert^2\,d\mu_t^j(x^j)\,dt\\
        &+  3\int_0^T \int_{\R^d}\left\vert \int_{\R^{d(m-1)}} b_j^i(x^j,y^{-j})\left( \tilde \rho_t^{x^j}(y^{-j})-\frac{\tilde\rho_t(x^j,y^{-j})}{\tilde \rho^j_t (x^j) \vee \varepsilon}\right)\,dy^{-j} \right\vert^2\,d\mu_t^j(x^j)\,dt\\
        &=3(\text{Term}_1+\text{Term}_2+\text{Term}_3).
    \end{split}
    \end{align}
    We break up the proof by analyzing the terms one by to get to the desired bound.
    We upper bound Term$_1$ by using Total Variation and Pinsker's inequality:
\begin{align}
\label{eq: term 1}
\begin{split}
    \text{Term}_1&=\int_0^T\int_{\R^d} \left\vert \int_{\R^{d(m-1)}} b_j^i(x^j,y^{-j}) \left(\mu_t^{x^j}-\tilde\mu^{x^j}_t\right)(y^{-j})\,dy^{-j}\right\vert^2 \,d \mu_t^j(x^{j})\,dt\\
    &\leq \|b_j^i\|_\infty^2 \int_0^T\int_{\R^d} \|\mu_t^{x^j}-\tilde\mu_t^{x^j}\|_{\text{TV}}^2\,d\mu_t^j(x^j)\,dt\\
    &\leq 2\|b_j^i\|_\infty^2\int_0^T \int_{\R^d} \ent(\mu_t^{x^j}\,\|\,\tilde\mu_t^{x^j})\,d\mu_t^j(x^{j})\,dt\\
    &\leq 2\|b_j^i\|_\infty^2\int_0^T\ent(\mu_t\,\|\,\tilde \mu_t)\,dt.
\end{split}
    \end{align}
Notice that we use the chain rule for relative entropy to go from $\int \ent(\mu_t^{x^j}\,\|\,\tilde\mu_t^{x^j})\,d\mu^j_t$ to $\ent(\mu_t\,\|\,\tilde\mu_t)$.
For Term$_2$, we use again a Total Variation bound to see
\begin{align*}
\text{Term}_2&=\int_0^T\int_{\R^d}\left\vert\int_{\R^{d(m-1)}} b_j^i(x^j,y^{-j})\left( \tilde \mu_t^{x^j}-\tilde\rho_t^{x^j}\right)(y^{-j})\,dy^{-j} \right\vert^2\,d\mu_t^j(x^j)\,dt\\
        &\leq \int_0^T\|b_j^i\|_\infty^2 \int_{\R^d} \|\tilde\mu_t^{x^j}-\tilde\rho_t^{x^j}\|_{\text{TV}}^2\,d\mu_t^j(x^j)\,dt\\
        &\leq\int_0^T \|b_j^i\|_\infty^2\left(\ent(\mu_t^j\,\|\,\tilde\mu_t^j)+\log \left(\int_{\R^d} \exp\left\{\|\tilde\rho_t^{x^j}-\tilde\mu_t^{x^j}\|_{\text{TV}}^2\right\}\,\,d\tilde \mu_t^j(x^j)\right)\right)\,dt
    \end{align*}
    where the second inequality is the variational representation of entropy. Since the total variation distance is bounded between $0$ and $2$, $0\leq \|\tilde\rho_t^{x^j}-\tilde\mu_t^{x^j}\|_{\text{TV}}^2\leq 4$; but for $s \in [0,4]$, $e^s\leq 1+s\cdot \frac{e^4-1}{4}$. As a result we obtain, for $c=(e^4-1)/4$
    \begin{align}
\label{eq: term 2}
\begin{split}
        \text{Term}_2&\leq \|b_j^i\|_\infty^2\int_0^T \left(\ent(\mu_t^j\,\|\,\tilde\mu_t^j)+\log \left(1+\frac{e^4-1}{4}\int_{\R^d}\|\tilde\rho_t^{x^j}-\tilde\mu_t^{x^j}\|^2_{\text{TV}}\,\tilde\mu_t^j({x^j})\,dx\right)\right)\,dt\\
        &\leq \|b_j^i\|_\infty^2\int_0^T \ent(\mu_t^j\,\|\,\tilde\mu_t^j)+ 2c\ent(\tilde \mu_t\,\|\,\tilde \rho_t)\,dt\leq \|b_j^i\|_\infty^2\int_0^T \ent(\mu_t\,\|\,\tilde\mu_t)+ 2c\ent(\mu_t\,\|\, \rho_t)\,dt.
\end{split}
    \end{align}
    The second inequality follows from $\log (1+x)\leq x$ and Pinsker's inequality and the chain rule for relative entropy. For the third, we use the Data Processing inequality on both entropic terms.\bigskip\\
    For the third term, we compute
\begin{align*}
\text{Term}_3&=\int_0^T\int_{\R^d}\left\vert \int_{\R^{d(m-1)}} b_j^i(x^j,y^{-j})\left( \tilde \rho_t^{x^j}(y^{-j})-\frac{\tilde\rho_t(x^j,y^{-j})}{\tilde \rho_t^j(x^j) \vee \varepsilon}\right)\,dy^{-j} \right\vert^2\,d\mu_t^j(x^j)\,dt\\
&=\int_0^T\int_{\tilde\rho_t^j(x^j)\leq \varepsilon} \left\vert\int_{\R^{d(m-1)}} b(x^j,y^{-j})\left( \tilde \rho_t^{x^j}(y^{-j})-\frac{\tilde \rho_t(x^j,y^{-j})}{\tilde \rho_t^j(x^j) \vee \varepsilon}\right)\,dy^{-j}\right\vert^2\,d\mu_t^j(x^j)\,dt\\
&\leq \int_0^T\|b_j^i\|_\infty^2 \mu_t^j(\tilde\rho_t^j\leq \varepsilon)\,dt.
\end{align*}

We repeatedly use the variational representation of relative entropy and the elementary inequality $\log(1+x)\leq x$ to estimate 
\begin{align*}
    \mu_t^j(\tilde\rho^j_t <\varepsilon)& \leq \ent(\mu_t^j\,\|\,\tilde\mu_t^j)+\log \left(\int_{\R^d} \exp\left\{\mathbf{1}_{\tilde\rho_t^j\leq \varepsilon}(x^j)\right\}\,d\tilde\mu_t^j(x^j)\right)\\
    &\leq \ent(\mu_t^j\,\|\,\tilde\mu_t^j)+\log \left(1+e\tilde\mu_t^j(\tilde \rho_t^j<\varepsilon)\right)\\
    &\leq \ent(\mu_t^j\,\|\,\tilde\mu_t^j) + e\tilde\mu_t^j(\tilde \rho_t^j<\varepsilon).
    \end{align*}
Repeat the same procedure again on $\tilde\mu_t^j(\tilde\rho_t^j<\varepsilon)$ to find
\begin{align*}
    \mu_t^j(\tilde\rho^j_t <\varepsilon)&\leq \ent(\mu_t^j\,\|\,\tilde\mu_t^j) + e\ent(\tilde\mu^j_t\,\|\,\tilde\rho_t^j)+e^2\tilde\rho_t^j(\tilde\rho_t^j < \varepsilon)\\
    &\leq \ent(\mu_t\,\|\,\tilde\mu_t)+e\ent(\mu_t\,\|\,\rho_t)+e^2\tilde\rho_t^j(\tilde\rho_t^j < \varepsilon);
\end{align*}
the last inequality follows  by the data processing inequality and the chain rule (which gives $\ent(\tilde \mu_t^j\,\|\,\tilde\rho_t^j)\leq\ent( \mu_t^j\,\|\,\rho_t^j)\leq \ent(\mu_t\,\|\,\rho_t)$). Finally we estimate, for an arbitrary $p\in (0,1)$ and $q\geq 1$,
\begin{align*}
    \tilde\rho_t^j(\tilde\rho_t^j < \varepsilon)&=\tilde \rho_t^j\left(\frac{1}{\tilde \rho_t^j}\geq \frac 1\varepsilon \right)\leq \varepsilon^{1-p} \int _{\R^d} (\tilde\rho^j_t)^{p}\,dx=\varepsilon^{1-p}\int_{\R^d} (\tilde \rho_t^j)^p \cdot \left(\frac{C+|x|^q}{C+|x|^q}\right)^{p}\,dx\\
    &\leq \varepsilon^{1-p}\left(p\int_{\R^d} (C+|x|^q) \tilde \rho_t^j(x)\,dx+(1-p)\int_{\R^d} \left(\frac{1}{C+|x|^q}\right)^\frac{p}{1-p}\,dx\right).
\end{align*}
The first inequality is by Markov's inequality, while the second the second is Young's inequality. This inequality holds for every $q\geq1, C>0$, and thus we can just take $q,C$ large enough that the second integral is finite, since $\tilde \rho_t^j$ has finite $q$-moment for every finite time $t$ and $q$, uniformly in $\varepsilon$ and $h$ (this follows by Assumptions \ref{assumption: initial regularity}(2), \ref{assumption: Kernel} and boundedness of the drift of \eqref{eq: intermediate problem}).
Putting everything together we find for every $p\in (0,1)$
\begin{equation}
\label{eq: term 3}
    \text{Term}_3\leq 4\|b^i_j\|_\infty^2 \left(\varepsilon^p C_{p,T}+ \int_0^T \ent(\mu_t\,\|\,\tilde\mu_t)+e\ent(\mu_t\,\|\,\rho_t)\,dt\right) 
\end{equation}
The constant $C_{p,T}$ can be taken to be
\[C_{p,T}=e^2\inf_{q,C}\left\{ (1-p)\int_0^T\int_{\R^d} (C+|x|^q) \tilde \rho_t^j(x)\,dx\,dt+Tp\int_{\R^d} \left(\frac{1}{C+|x|^q}\right)^\frac{1-p}{p}\,dx\right\},\]
We now go back to Equation \eqref{eq: bound on a}. Using the bounds \eqref{eq: term 1},\eqref{eq: term 2},\eqref{eq: term 3} and changing $C$ from line to line, we arrive at
\begin{align*}
\int_0^T &\int_{\R^d} \left|\int_{\R^{d(m-1)}}  b_j^i(x^j,y^{-j})\left(\frac{\tilde\rho_t(x^j,y^{-j})}{\tilde \rho^j_t (x^j) \vee \varepsilon} -\mu_t^{x^j}(y^{-j})\right)\,dy^{-j}\right|^2\,d\mu_t^j(x^j)\,dt
\\&\leq C\varepsilon^p+ C\int_0^T \ent(\mu_t\,\|\,\tilde\mu_t)\,dt+C\int_0^T\ent(\mu_t\,\|\,\rho_t)\,dt\\
&\leq C\varepsilon^p + C h^2+C\int_0^T\ent(\mu[t]\,\|\,\rho[t])\,dt,
\end{align*}
with the last inequality following from Lemma \ref{lm: bound on time integral of relative entropy} (for the time integral of $\ent(\mu_t\,\|\,\tilde \mu_t)$) and the chain rule. By a symmetric line of reasoning on the other $b^i_j$ terms, we arrive at
\begin{equation*}
    \ent \left(\mu[T]\,\|\,\rho[T]\right) \leq \varepsilon^p C+ Ch^2+C\int_0^T\ent(\mu[t]\,\|\,\rho[t])\,dt
\end{equation*}
By Gronwall we conclude that for some constant $C$ that is allowed to depend on $p,b,\mu_0,\sigma,m,d$, we obtain
\[\ent \left(\mu[T]\,\|\,\rho[T]\right)\leq C\left(h^2 +\varepsilon^p\right).\]
\end{proof}
We conclude by putting together all previous estimates. 
    \begin{proof}[Proof of Theorem \ref{th: main}]
    By the triangle inequality and Pinsker's inequality we have
    \begin{align*}
        \|\pi^k[t]-\mu^{\otimes k}[t]\|_{\textup{TV}}&\leq\|\pi^k[t]-\rho^{\otimes k}[t]\|_{\textup{TV}}+\|\rho^{\otimes k}[t]-\mu^{\otimes k}[t]\|_{\textup{TV}}\\
        &\leq \sqrt {2\ent (\pi^k[t]\,\|\,\rho^{\otimes k}[t])}+\sqrt {2\ent (\mu^{\otimes k}[t]\,\|\,\rho^{\otimes k}[t])}
    \end{align*}
    Since $\rho_t^{\otimes k}$ and $\mu^{\otimes k}$ are product measure, the tensorization property of relative entropy yields
    \begin{equation*}
            \|\pi^k[t]-\mu^{\otimes k}[t]\|_{\textup{TV}}\leq C\left( \sqrt{\ent(\pi^k[t]\,\|\,\rho^{\otimes k}[t])}+\sqrt{k \ent (\mu[t]\,\|\,\rho[t])}\right) 
        \end{equation*}
        By putting together the estimates from Proposition \ref{eq: regularized process to MKV} and Proposition \ref{pr: particles to intermediate}, for every $r<1$ we have
        \begin{equation*}
            \|\pi^k[t]-\mu^{\otimes k}[t]\|_{\textup{TV}}\leq C\sqrt k\left(\frac{e^{\frac{C}{\varepsilon^2 h^d}}}{\varepsilon \sqrt{h^d}\sqrt{n }}+\varepsilon^{r\slash2}+h\right)
        \end{equation*}
        Setting $\varepsilon= h^{2/r}\sqrt C $, $h=(\frac{1}{4}\log n)^{-\frac{r}{dr+4}}$ yields
        \begin{equation*}
            \|\pi^k[t]-\mu^{\otimes k}[t]\|_{\text{TV}}\leq C \frac{\sqrt k}{(\log n)^{\frac{r}{dr+4}}}.
        \end{equation*}
    \end{proof}

\section{Regularity Estimates and Proof of Lemma \ref{lm: bound on time integral of relative entropy}}
\label{sect: regularity estimates}
The goal of this Section is to prove the technical Lemma \ref{lm: bound on time integral of relative entropy}. To control $\ent (\mu_t\,\|\,\mu_t \star_j K_h)$, we need to interpolate between $\log \mu_t$ and $\log \mu_t \star_j K_h$; to do so effectively we need $\log \mu_t$ to have some Sobolev regularity. In particular, we will show that controlling $\ent (\mu_t\,\|\,\mu_t \star_j K_h)$ reduces to checking that $\log \mu_t$ has two weak derivatives that satisfy $\int |\nabla^2 \log \mu_t|^3\,d\mu_t<\infty$.
Since a priori it is not clear that $\mu_t$ is regular enough, we need to perform a mollification procedure to justify calculations rigorously. Once we check that $\log\mu_t$ belongs to $H^2_{\text{loc}}(\R^{dm})$, we obtain further regularity of the non-linear coefficient, which in turn allows us to close the proof. \\

Our approach is related to the techniques developed in \cite[Chapter 7.4]{bogachev2022fokker}, but at higher orders. The estimates tread a similar path as \cite[Appendix A]{Chaintron2024}, especially for the control of $\nabla^3 \log\mu_t$, but our set-ups are quite different: they work with non identity diffusion but can rely on global Lipschitz assumptions, while our coefficients are a priori only in $L^\infty(\R^d)$. As a result we do not use a time reversal argument, but work at the level of the mollified density and bootstrap our way to higher regularity. From now on, we adopt the more compact notation 
\begin{equation*}
    dX_t= f_t (X_t)\,dt+\sqrt2\sigma\,dW_t
\end{equation*}
for Equation (\ref{eq: main SDE}), with $f_t=(f_t^i)_{i=1}^m$ and
\begin{align*}
    f^i_t(x)=
        \sum_{j=1}^m \E\left[b^i_j(X_t) \,\big\vert\, X_t^j=x^j\right]+V^i(x).
\end{align*}

\subsection{Preliminary regularity and mollification}
The fact that $f_t$ is bounded and assumption \ref{assumption: initial regularity} provide some preliminary regularity using results from \cite{bogachev2022fokker}, which we summarize here. In particular, we can obtain that $f_t$ is in fact weakly differentiable.
\begin{lemma}[Preliminary regularity] Under assumption \ref{assumption: initial regularity}, the following holds.
\label{Lm: derivatives of CE} \begin{enumerate}
    \item For almost every $t\geq 0$, $\mu_t \in W^{1,1}(\R^d)$. Moreover $\mu_t$ is locally bounded away from $0$ for every $t\geq 0$.
    \item Fisher information is finite: for almost every $t>0$, $\log \mu_t \in H^1_{\textup{loc}}(\R^{dm})$ and 
    \begin{equation}
        \label{eq: finite fisher info} \int_0^T\int_{\R^{dm}}|\nabla \log \mu_t(x)|^2\mu_t(x)\,dx\,dt<\infty.
    \end{equation}
    \item For almost every $t>0$, $\mu_t \in H^1(\R^d)$. Moreover  \[\|\mu\|_{L^\infty(\R^{dm}\times [0,T])}<\infty.\]
\item for almost every $t$, $f_t \in H^{1}_\textup{loc}(\R^{dm})$ and 
    \begin{equation}
    \label{eq: finite second derivative of f}
        \sum_{i=1}^m\int_0^T \int_{\R^{dm}}|\partial_i f_t( x)|^2\mu_t(x)\,dx\,dt<\infty.
    \end{equation} 
\end{enumerate}
\end{lemma}
\begin{proof}
\begin{enumerate}
\item By Assumption \ref{assumption: initial regularity}(3), $\int |\log \mu_0|\,d\mu_0<\infty$; since $f_t$ is uniformly bounded, the assumptions of \cite[Theorem 7.4.1]{bogachev2022fokker} are satisfied and we conclude that $\mu_t \in W^{1,1}(\R^d)$ for almost every $t>0$. As a result, we can apply \cite[Theorem 8.2.1]{bogachev2022fokker} to obtain that $\mu_t$ is locally bounded away from zero for every $t>0$; the case $t=0$ is assumed.
\item  By \cite[Theorem 7.4.1]{bogachev2022fokker}, 
    \[\int_0^T\int_{\R^{dm}} \left|\frac{\nabla \mu_t(x)}{\mu_t(x)}\right|^2\mu_t(x)\,dx\,dt<\infty.\]
    Moreover, since $\mu_t$ is locally bounded away from zero, we can use the usual chain rule for weak derivatives to obtain that $\log\mu_t \in H^{1}_{\textup{loc}}(\R^{dm})$ and $\nabla\log\mu_t= \nabla\mu_t /\mu_t$.
    \item Boundedness of the drift and \ref{assumption: initial regularity}(4)  imply that $\|\mu\|_{L^\infty(\mathbb{R}^{dm}\times[0,T])}<\infty$ by \cite[Corollary 7.3.8]{bogachev2022fokker}. Then, 
    \begin{align*}
        \int_0^T \int_{\R^{dm}}|\nabla \mu_t(x)|^2\,dx&\leq \|\mu\|_{L^\infty(\mathbb{R}^{dm}\times[0,T])}\int_0^T \int_{\R^{dm}}|\nabla\log\mu_t(x)|^2\mu_t(x)\,dx\,dt<\infty.
    \end{align*}
    Hence for almost every $t>0$, $\mu_t \in H^1(\R^d)$. 
    \item Differentiability of $V^i$ is assumed. Since $\mu_t$ can be taken to be a locally H\"older continuous density function, we can identify the conditional expectation terms with their continuous version
    \begin{equation*}
    \E\left[b_j^i(X_t)\mid X^j_t=x^j\right]=\frac{1}{\mu_t^j(x^j)}\int_{\R^{d(m-1)}} b^i_j(x^j,y^{-j})\mu_t(x^j,y^{-j})\,dy^{-j}.
    \end{equation*}
     The marginal density $\mu_t^j$ must be locally bounded away from zero by point (1); moreover since $\nabla \log\mu_t^j(X_t^j)=\E[\nabla_j \log \mu_t(X_t)\mid X_t^j]$ we get
     \[\E[|\nabla \log \mu_t^j(X^j_t)|^2]=\E\left[\left|\E[\nabla_j \log \mu_t(X_t)\mid X_t^j]\right|^2\right]<\infty\] 
     for almost every $t$. On the other hand, since $b^i_j$ is bounded with bounded derivatives and $|\nabla \log \mu_t| \in L^2(\mu_t)$ by the previous point, we can use Lemma \ref{lm: derivative under integral} to show that in the sense of weak derivatives 
    \begin{align*}
    \partial_{x_k}\int b^i_j(x^j,y^{-j})\mu_t(x^j,y^{-j})\,dy=\int \partial_{x_k}(b^i_j(x^j,y^{-j})\mu_t(x^j,y^{-j}))\,dy.    
    \end{align*}
    We use the chain rule for weak derivatives to check that $\nabla (1/\mu_t)= -\nabla\mu_t / (\mu_t)^2$ since $\mu_t>0$; the above discussion implies that
    \begin{align*}\int b^i_j(\cdot,y^{-j})\mu_t(\cdot,y^{-j})\,dy\,,\,\partial_{x_k}\int b^i_j(\cdot,y^{-j})\mu_t(\cdot,y^{-j})\,dy,\,\frac{1}{\mu_t^j}, \nabla\left( 
        \frac{1}{\mu_t^j}\right) \in L^2_{\text{loc}}(\R^d).
    \end{align*}
    We can then employ the product rule for weak derivatives, and thus obtain
    \begin{align*}
            \partial_{x_k}\left(\frac{1}{\mu_t^j(x^j)}\int b^i_j(x^j,y^{-j})\mu_t(x^j,y^{-j})\,dy\right)&=-\partial_{x_k} \log \mu_t^j(x^j) \E[b^i_j(X_t)\mid X^j_t=x^j] \\
            &+ \E[\partial_{x_k} b^i_j(X_t)\mid X^j_t=x^j]\\
        &+ \E[b^i_j(X_t)\partial_{x_k}\log \mu_t(X_t)\mid X^j_t=x^j]\\
        &\leq C\left(1+|\partial_{x_k}\log \mu_t^j(x^j)|+\E[|\partial_{x_k} \log \mu_t(X_t)|\mid X^j_t]\right).
    \end{align*}
    The integral bound \eqref{eq: finite second derivative of f} follows by Jensen inequality and by the bound \eqref{eq: finite fisher info}.
\end{enumerate}

\end{proof}

To prove the required regularity of the density $\mu_t$, we first perform a mollification procedure in space and in time;  we need to do so because a priori, the density $\mu_t$ does not have enough regularity (neither in time nor space) to compute the time derivatives that will be the backbone of our approach. This procedure was pioneered by \cite{DiPernaLions} and is by now common in the literature on SDEs with rough coefficients, appearing for instance in \cite{FigalliRoughSDE}, \cite{Trevisan_superposition} and \cite{Ambrosio2008-dt}. One key point already observed in \cite{Trevisan_superposition} is that the mollification procedure preserves integral and uniform bounds, and thus does not change our control on the conditional expectation term; moreover, it preserves the control over integrals of the drift and its derivatives.
In what follows, recall that $\mu_t$ is a weak solution to the Fokker--Planck PDE
\begin{align*}
    \partial_t \mu_t = -\text{div}(f_t\mu_t)+\sigma^2 \Delta\mu_t,
\end{align*}
meaning that for every compactly supported, smooth function $\varphi:[0,T]\times \R^{dm}\to \R $ it holds 
\begin{equation*}
    \int_0^T \int_{\R^{dm}} \left(\partial_t \varphi_t(x)+ \nabla \varphi_t(x)\cdot f_t(x)+\sigma^2 \Delta \varphi_t(x)\right)\mu_t(x)\,dx\,dt=0.
\end{equation*}
Let $\eta \in C^\infty_b(\R^d)$ be a probability measure with $|\nabla^k \eta|\leq C_k \eta $ for every $k$ and denote $\eta^r(x)=\frac{1}{r^{dm}}\eta(x\slash r)$ for some $r>0$. Consider the probability density 
        \begin{equation*}
            \mu_t^r=\mu_t\star \eta^r=\int_{\R^{dm}} \mu_t(x-rz)\eta(z)\,dz.
        \end{equation*}
Using the fact that derivatives and convolutions commute, we obtain that $\mu^r_t$ solves (in the weak sense) the PDE
    \begin{align}
    \label{eq: PDE r}
    \begin{split}
        \partial_t \mu_t^r &= -\text{div}(\mu_t^r f^r_t)+\sigma^2  \Delta \mu_t^r,\\
        \mu_0^r&=\mu_0\star \eta^r.
    \end{split}
    \end{align}
    with $f_t^r= (f_t\mu)\star \eta^r \slash \mu^r$; from a probabilistic perspective, we can interpret the new mollified drift as 
    \begin{equation*}
        f_t^r( x)=\E[f_t({X}_t)\mid {X}_t+rZ= x],
    \end{equation*}
    where $Z$ is independent of $X_t$ and $\text{Law}(Z)=\eta$. Since $f_t$ is bounded, $f_t^r \in C_b^\infty(\R^d)$ (as was already observed in \cite[Theorem 2.5]{FigalliRoughSDE}).
    Now, take $\theta\in C_b^\infty(\R)$ to be again a probability measure supported on $(0,1)$. We further mollify $\mu_t^r$ in time: for $t\geq q$, define 
    \begin{equation*}
        \mu_t^{r,q}( x)=\int_0^1\mu^r_{t-sq}( x)\theta(s)\,ds;
    \end{equation*}
    we will denote by $\star_t$ the operation of convolution in the time variable. Let $\ell\geq q$; we conclude that $\mu_t^{r,q}$ is a classical (and smooth) solution of the PDE
    \begin{align}
        \label{eq:PDE q}
    \begin{split}
        \partial_t \mu_t^{r,q}&=-\text{div} (\mu_t^{r,q} f_t^{r,q})+\sigma^2 \Delta\mu_t^{r,q}\\
        \mu^{r,q}_{\ell}&=\mu^r_{\ell} \star_t \theta^q
    \end{split}
    \end{align}
    with
    \begin{equation*}
        f_t^{r,q}=\frac{(f_t^r \cdot \mu_t^r)\star_t \theta^q}{\mu_t^{r,q}}.
    \end{equation*}
    To be clear, we are starting the PDE from the initial distribution $\mu_\ell^{r,q}$, with $\ell\geq q>0$; this is to make sure that the integral defining $\mu_t^{r,q}$ does not involve $\mu_t^{r}$ for $t<0$.

We collect some useful properties of the mollified densities $\mu^r,\mu^{r,q}$ and of the coefficients $f^{r,q},f^r$; we will state them in this Section but postpone their proofs to Appendix \ref{Appendix: mollification estimates}.
\begin{lemma}[Properties of Mollified Densities]
\label{lm: properties of mollified densities} The following holds:
    \begin{enumerate}
        \item For every $t\geq 0$ and each $r>0$ , $q\geq 0$ and $k\in \N$, there exists a constant $C_r^k$ not depending on $q$ such that
        \begin{align*}
        |\nabla^k\log \mu_t^r|+|\nabla^k \log \mu_t^{r,q}|\leq C^k_r.
        \end{align*}
        \item For every $t\geq 0, r>0$ and $ q>0$, $f_t^r,f_t^{r,q} \in C_b^3(\R^d)$.
        \item The density $\mu_t^r$ and its space derivatives $\partial_i\mu_t^r,\partial_j \partial_i\mu_t^r,\partial_l\partial_j\partial_i \mu_t^r$ are absolutely continuous in time.
        \item If $\mu_t \in H^1(\R^d)$, for every $p\geq1$, $n \in \N$, $i\in [md]$ we have
        \begin{equation*}
            \int_{\R^{dm}}|\partial_i^n \mu_t^{r}(x)/\mu_t^{r}(x)|^p\mu_t^r(x)\,dx\leq \int_{\R^{dm}}|\partial_i^n  \mu_t(x)/\mu_t(x)|^p\mu_t(x)\,dx.
        \end{equation*}
        \item For each $i$, each $\ell\geq q>0$ and $p\geq1$ we have 
        \[\int_\ell^T \int_{\R^{dm}}|\partial_i\log\mu^{r,q}_t(x)|^p\mu_t^{r,q}(x)\,dx\,dt\leq \int_0^T \int_{\R^{dm}}|\partial_i \log \mu_t(x)|^p \mu_t(x)\,dx.\]
        \item For each $i$ and $p\geq 1$ we have
    \begin{align*}
        \int_{\ell}^T\int_{\R^{dm}} |\partial_i f_t^{r,q}(x)|^p\mu_t^{r,q}(x)\,dx\,dt&\leq \int_{0}^T\int_{\R^{dm}} C^p(1+|\partial_i \log \mu_t(x)|^p)\mu_t(x)\,dx\,dt
    \end{align*}
    Where $C$ does not depend on $r,q$ or $\ell$.
    \end{enumerate}
\end{lemma}
     \begin{remark}
         In what follows we will perform integration by parts repeatedly. These operations can be justified rigorously by the boundedness of all the involved functions.
     \end{remark}
\subsection{Second derivatives and bootstrapping} We now proceed with the energy estimates. The idea is to compute the time derivatives of $\int|\nabla \log \mu_t^{r,q}|^2\,d\mu_t^{r,q}$, obtain bounds that do not depend on the mollification parameters and then pass to the limit as the regularization converges to $0$; this will be made possible by the conditional expectation structure embedded in Lemma \ref{lm: properties of mollified densities}(6).

\begin{proposition}[Second derivatives, square moments]
\label{pr: second derivatives square moments}
    Under assumption \ref{assumption: initial regularity}, it holds that for almost every $t>0$, $\mu_t \in H^2(\R^d)$, $\log \mu_t \in H^2_\textup{loc}(\R^d)$ and 
    \begin{equation*}
    \int_0^{T}\int_{\R^{dm}} \left(\|\nabla ^2\mu_t(x)/\mu_t(x)\|_{\textup{ Fr}}^2+|\nabla^2 \log \mu_t(x)|^2+|\nabla\log\mu_t(x)|^4\right)\mu_t(x)\,dx\,dt \leq C<\infty.
    \end{equation*}
\end{proposition}
\begin{proof}
Recall that $\mu_t^{r,q}$ is a smooth solution to the PDE \begin{align*}
    \partial_t \mu^{r,q}_t &= -\text{div} (f^{r,q}_t\, \mu^{r,q}_t) +\sigma^2 \Delta \mu^{r,q}_t
\end{align*}
And as a consequence,
    \begin{align*}
    \partial_t \log \mu^{r,q}_t &= -\text{div} f^{r,q}_t - \nabla \log \mu^{r,q}_t \cdot f^{r,q}_t +\sigma^2 \frac{\Delta \mu^{r,q}_t}{\mu^{r,q}_t} \\
    &= -\text{div} f^{r,q}_t - \nabla \log \mu^{r,q}_t \cdot f^{r,q}_t +\sigma^2 \Delta \log \mu^{r,q}_t + \sigma^2| \nabla \log \mu^{r,q}_t|^2.
\end{align*}
Differentiate in space to obtain
\begin{equation}
\label{eq: HJB PDE first derivative}
\partial_t 
    \partial_i\log \mu^{r,q}_t = -\partial_i\text{div} f^{r,q}_t - \partial_i\nabla \log \mu^{r,q}_t \cdot f^{r,q}_t-\nabla \log \mu^{r,q}_t \cdot \partial_i f^{r,q}_t +\sigma^2 \partial_i\Delta \log \mu^{r,q}_t + \sigma^2 \partial_i| \nabla \log \mu^{r,q}_t|^2
\end{equation}
Thanks to the boundedness and smoothness of $\nabla^{p} \log \mu_t^{r,q}$ for every $p\geq 1$ (from Lemma \ref{lm: properties of mollified densities}(1)), we can use Leibniz's rule to pass the derivative under the integral to obtain
\begin{align*}
    \frac{d}{dt}&\int_{\R^{dm}}|\partial_i \log \mu_t^{r,q}(x)|^2\,d\mu_t^{r,q}(x)=-\int_{\R^{dm}}|\partial_i \log \mu_t^{r,q}(x)|^2\text{div} (f_t^{r,q}(x) \mu_t^{r,q}(x)-\sigma^2 \nabla \mu_t^{r,q}(x))\,dx.\\
    & -2\int_{\R^{dm}} \partial_i\log \mu_t^{r,q}(x)\left(\partial_i\text{div} f^{r,q}_t(x) + \partial_i\nabla \log \mu^{r,q}_t(x) \cdot f^{r,q}_t(x)+\nabla \log \mu^{r,q}_t(x) \cdot \partial_i f^{r,q}_t(x)\right)\,d\mu_t^{r,q}(x)\\
    &+\sigma^2 \int_{\R^{dm}} \partial_i \log \mu_t^r(x) (\partial_i\Delta \log \mu_t^{r,q}(x)+\partial_i |\nabla \log \mu_t^{r,q}(x)|^2)\,d\mu_t^{r,q}(x).
\end{align*}
Integrating by parts the divergence terms and the Laplacian we get
\begin{align*}
   \frac{d}{dt}&\int_{\R^{dm}}|\partial_i \log \mu_t^{r,q}(x)|^2\,d\mu_t^{r,q}(x)=2\int_{\R^{dm}}\partial_i \log \mu_t^{r,q}(x)\partial_i\nabla \log \mu_t^{r,q}(x)(f_t^{r,q}(x) -\sigma^2 \nabla \log \mu_t^{r,q}(x))\,dx.\\
   &+2\int_{\R^{dm}} \nabla \partial_i \log \mu_t^{r,q}(x) \cdot \partial_i f_t^{r,q}(x)+\partial_i \log\mu_t^{r,q}(x)\partial_i f_t^{r,q}(x)\cdot \nabla \log \mu_t^{r,q}(x)\,d\mu_t^{r,q}(x)\\
    &-2\int_{\R^{dm}}\partial_i\log \mu_t^{r,q}(x)\left( \partial_i\nabla \log \mu^{r,q}_t(x) \cdot f^{r,q}_t(x)+\nabla \log \mu^{r,q}_t(x) \cdot \partial_i f^{r,q}_t(x)\right)\,d\mu_t^{r,q}(x)\\
    &-2\sigma^2 \int_{\R^{dm}} |\partial_i \nabla \log \mu_t^{r,q}(x)|^2+\partial_i \log \mu_t^{r,q}(x) \partial_i \nabla \log \mu_t^{r,q}(x)\cdot \nabla \log \mu_t^{r,q}(x)\,d\mu_t^{r,q}(x)\\
    &+2\sigma^2 \int_{\R^{dm}} \partial_i \log \mu_t^{r,q}(x) \partial_i |\nabla \log \mu_t^{r,q}(x)|^2\,d\mu_t^{r,q}.
\end{align*}
After performing the required cancellations, we are left with
    \begin{align*}
    \frac{d}{dt}\int_{\R^{dm}}|\partial_i \log \mu_t^{r,q}(x)|^2\,d\mu_t^{r,q}(x)&=2\int_{\R^{dm}} \nabla \partial_i \log \mu_t^{r,q}(x) \cdot \partial_i f_t^{r,q}(x)\,d\mu_t^{r,q}(x)-2\sigma^2 \int_{\R^{dm}} |\partial_i \nabla \log \mu_t^{r,q}(x)|^2\,d\mu_t^{r,q}(x)\\
    &\leq 2\int_{\R^{dm}}\left(\frac{1}{2\zeta} |\partial_i f_t^{r,q}(x)|^2+\left(\frac{\zeta}{2}-\sigma^2\right) |\partial_i \nabla \log \mu_t^{r,q}(x)|^2\,d\mu_t^{r,q}(x)\right) \,d\mu_t^{r,q}(x),
\end{align*}

by using Young's inequality in the last line. Set $\zeta$ small enough to find that for some positive constants $C_1,C_2$ that do not depend on the mollification parameters,
\begin{align*}
\begin{split}
\int_{\R^{dm}}|\partial_i \log \mu_t^{r,q}(x)|^2\,d\mu_t^{r,q}(x) +  &C_1\int_\ell^T \int_{\R^{dm}} |\partial_i \nabla \log \mu_t^{r,q}(x)|^2\,d\mu_t^{r,q}(x)\,dt\\
&\leq \int_{\R^{dm}} |\partial_i \log \mu_\ell^{r,q}(x)|^2\,d\mu_\ell^{r,q}(x) +C_2 \int_\ell^T \int_{\R^{dm}} |\partial_i f_t^{r,q}(x)|^2\,d\mu_t^{r,q}(x)\,dt. 
\end{split}
\end{align*}
The integral of $\partial_i f_t^{r,q}$ can be upper bounded independently of $\ell, q,r$ by Lemmas \ref{lm: properties of mollified densities} (6) and \ref{Lm: derivatives of CE}(2), and we thus replace it by a constant. We obtain that for some $C$ independent of the mollification parameters
    \begin{align}
\label{eq: removing mollification q}
\begin{split}
\int_{\R^{dm}}|\partial_i \log \mu_t^{r,q}(x)|^2\,d\mu_t^{r,q}(x) +  &C_1\int_\ell^T \int_{\R^{dm}} |\partial_i \nabla \log \mu_t^{r,q}(x)|^2\,d\mu_t^{r,q}(x)\,dt\\
&\leq \int_{\R^{dm}} |\partial_i \log \mu_\ell^{r,q}(x)|^2\,d\mu_\ell^{r,q}(x) +C_2. 
\end{split}
\end{align}
We now take the limit as $q\to 0$, beginning from the right hand side of \eqref{eq: removing mollification q}. By Lemma \ref{lm: properties of mollified densities}(1) $\partial_i \log \mu_\ell^{r,q}$ is bounded uniformly in $q$; moreover by Lemma \ref{lm: properties of mollified densities}(3) $\mu_\ell^{r,q}(x)\to \mu_\ell^r(x)$ for every $x$ and $\partial_i \log \mu_\ell^{r,q}(x)\to \partial_i\log \mu_\ell^r(x)$ again everywhere. We conclude using the Dominated Convergence Theorem that as $q \downarrow 0$,
\begin{align*}
    &\lim _{q \to 0} \bigg\vert \int_{\R^{dm}}|\partial_i \log \mu_\ell^{r,q} (x)|^2\,d\mu^{r,q}_\ell(x)- \int_{\R^{dm}}|\partial_i \log \mu_\ell^r (x)|^2\,d\mu^r_\ell(x)\bigg\vert\\
    &\leq \lim _{q \to 0} \bigg\vert\int_{\R^{dm}}\left(|\partial_i \log \mu_\ell^{r,q} (x)|^2-|\partial_i \log \mu_\ell^r(x)|^2 \right)\,d\mu^r_\ell(x)\bigg\vert+\left\vert\int_{\R^{dm}} |\partial_i \log \mu_\ell^{r,q} (x)|^2(\mu_\ell^{r,q}(x)-\,d\mu_\ell^r(x))\right\vert\\
    &\leq 0 +C \lim_{q\to 0} \|\mu^{r,q}_\ell-\mu^r_\ell\|_{L^1 (\R^d)}=0.  
\end{align*}
To see convergence in total variation of $\mu_\ell^{r,q} $ to $\mu_\ell^r$ we can use Scheffé's Lemma since $\mu_\ell^{r,q}\to \mu_\ell^{r}$ almost everywhere by continuity. Since by Lemma \ref{lm: properties of mollified densities}(3), $t\mapsto \partial_i\partial_j \mu_t^{r}$ is continuous, we know that $\partial_{i}\partial_j \mu_t^{r,q}$ converges to $\partial_{i}\partial_j \mu_t^{r,q}$ as $q\to0$ and thus we can use Fatou's lemma on the left hand side of \eqref{eq: removing mollification q}. Since we are imposing $\ell\geq q>0$, we get for every $\ell >0,$
\begin{align*} \int_{\R^{dm}}|\partial_i \log \mu^r_T(x)|^2\,d\mu_T^r(x)&+ C_1\int_\ell^T\int_{\R^{dm}} |\partial_i \nabla \log \mu_t^r (x)|^2\,d\mu_t^r(x)\,dt
    \\
    &\leq\int_{\R^{dm}}|\partial_i \log \mu_\ell^r (x)|^2\,d\mu_\ell^r(x) + C_2.
\end{align*}
We now push $\ell \downarrow 0$. The limit on the left hand side follows from the Monotone Convergence Theorem. On the right hand side, we have that $\mu_\ell^r\to \mu_0^r$ everywhere by continuity; the same applies to $\partial_i \mu_t^r (X_t)$ by Lemma \ref{lm: properties of mollified densities}(3). Since $|\partial_i \log \mu_\ell^r|$ is bounded uniformly in $\ell$ by Lemma \ref{lm: properties of mollified densities}(1) and $\mu_\ell^r\to \mu_0^r$ in Total Variation by Lemma \ref{lm: properties of mollified densities}(3) and Scheffé's Lemma, we can repeat the procedure used to take $q\to 0$ to get $\int |\partial_i \log \mu_\ell^r|^2\,d\mu_\ell^r\to \int |\partial_i \log \mu_0^r|^2\,d\mu_0^r$; we obtain
\begin{align*}
\label{eq: bounds on second derivatives r}
\begin{split}
\int_{\R^{dm}}|\partial_i \log \mu^r_T(x)|^2\,d\mu_T^r(x)&+ C_1\int_0^T\int_{\R^{dm}} |\partial_i \nabla \log \mu_t^r (x)|^2\,d\mu_t^r(x)\,dt
    \\
    &\leq\int_{\R^{dm}}|\partial_i \log \mu_0^r (x)|^2\,d\mu_0^r(x) + C_2.
\end{split}
\end{align*}

We can use Lemma \ref{lm: properties of mollified densities}(4) to control the right hand side, since by assumption \ref{assumption: initial regularity}(3) $\mu_0 \in W^{1,1}(\R^d)$; we obtain $
    \int |\partial_i \log \mu_0^r|^2\,d\mu_0^r\leq \int|\partial_i \log \mu_0|^2\,d\mu_0<\infty.$ Now the right hand side is bounded uniformly in $r$; since $i$ is arbitrary we conclude
\begin{equation}
\label{eq: second moments second derivatives r}
    \sup_{r \leq 1} \int_0^T \int_{\R^{dm}} \|\nabla^2 \log \mu^r_t(x)\|^2_{\textup{Fr}}\,d\mu_t^r(x)\,dt <\infty.
\end{equation}
Integrating by parts and using Young's inequality we can check that 
\begin{align*}
    \int_{\R^{dm}} |\partial_i \log \mu^r_t(x)|^4\,d\mu_t^r(x)\,&= -3\int_{\R^{dm}} |\partial_i \log \mu_t^r(x)|^2 \partial^2_i \log \mu^r_t(x)\,d\mu_t^r(x)\\
    &\leq \frac{3\zeta}{2}\int_{\R^{dm}} |\partial_i \log \mu^r_t(x)|^4 \,d\mu_t^r(x)+\frac{3}{2\zeta}\int_{\R^{dm}} |\partial_i^2 \log \mu^r_t(x)|^2\,d\mu_t^r(x) 
\end{align*}
Taking $\zeta$ small enough and integrating in time we find
\begin{align}
\label{eq: bound on fourth derivatives}
    \sup_{r \leq 1}\int_0^T\int_{\R^{dm}} |\partial_i \log \mu^r_t(x)|^4 \,d\mu_t^r(x)\,dt\leq C \sup_{r \leq 1 }\int_0^T\int_{\R^{dm}} |\partial^2_i \log \mu^r_t(x)|\,d\mu_t^r(x)\,dt<\infty.
\end{align}
The estimate holds for all $i$, and we obtain
\begin{align*}
    \sup_{r\leq 1}\int_0^T \int_{\R^{dm}} &\|\nabla^2 \mu_t^r(x)\|^2_{\textup Fr}\,dx\,dt\\
    &\leq 4\|\mu\|_{L^\infty(\R^{dm}\times(0,T))}\sup_{r\leq 1}\int_0^T \int_{\R^{dm}}\left( \left\|\nabla^2 \log\mu_t^r(x)\right\|_{\text{Fr}}^2+|\nabla \log \mu_t^r(x)|^4 \right)\,d\mu_t^r(x)\,dt<\infty,
\end{align*}
(notice that we are using $\|\mu^r\|_{L^\infty(\R^d\times[0,T])}\leq \|\mu\|_{L^\infty(\R^d\times[0,T])}<\infty$ and Lemma \ref{Lm: derivatives of CE}(3)) which in turn implies that  for almost every $t>0$, $\mu_t \in H^2(\R^d)$ by a common weak compactness argument. As a result we obtain $\log\mu_t \in H_{\text{loc}}^2(\R^d)$ for almost every $t>0$ since $\mu_t$ is bounded away from zero on compact sets.
Now we can use Fatou's lemma to let $r\to 0$ and find the required bounds from \eqref{eq: bound on fourth derivatives} and \eqref{eq: second moments second derivatives r}.
\end{proof}

A key point is that the above estimates generate more regularity for equation \eqref{eq: main SDE}. The fact that $\mu_t \in H^2(\R^d)$ for almost every $t$ implies that $f_t$ also has two weak derivatives; our control on $\E[|\partial_i\partial_j \log \mu_t(X_t)|^2]$ allows us to control the moments of $\partial_i\partial_j f_t( X_t)$, at least when integrating in time.

\begin{proposition}[Bootstrapping]
\label{pr: second derivatives of drift}
     For almost every $t>0$, $f_t \in H^2_{\textup{loc}}(\R^d)$ and for each $i$,
        \begin{equation*}
            \int_0^T \int_{\R^{dm}}|\partial^2_i f_t(x)|^2\,d\mu_t(x)\,dt\leq C<\infty.
        \end{equation*}
\end{proposition}
\begin{proof}
  The functions $V^i$ have bounded derivatives.  For the conditional expectations components, we follow the same line of reasoning as Lemma \ref{Lm: derivatives of CE}. Since for almost every $t\geq 0,\mu_t \in H^2(\R^d)$ and 
        \[\int_{\R^{dm}}|\partial_i \log \mu_t^j(x)|^2\,d\mu^j_t(x)+\int_{\R^{dm}}|\partial_i^2 \log \mu_t( x)|^2\,d\mu_t(x)<\infty,\]
        we can pass the derivative under the integral and apply the product rule for weak derivatives. We then obtain
    \begin{align*}
        \partial^2_{x_k}&\left(\frac{1}{\mu_t^j(x^j)}\int b^i_j(x^j,y^{-j})\mu_t(x^j,y^{-j})\,dy\right)\\
        &=-\partial^2_{x_k} \log \mu_t^j(x^j) \E[b^i_j(X_t) \mid X^j_t=x^j]-\partial_{x_k}\log\mu_t^j(x^j)\partial_{x_k}\E[b^i_j(X_t)\mid X^j_t=x^j]\\
        &+\partial_{x_k} \E[\partial_{x_k}b^i_j(X_t) \mid X^j_t=x^j]-\partial_{x_k} \log \mu_t^j(x^j)\E[b^i_j(X_t)\partial_{x_k}\log \mu_t(X_t) \mid X^j_t=x^j]\\
    &+\E[\partial_{x_k}b^i_j(X_t)\partial_{x_k}\log \mu_t(X_t)\mid X^j_t=x^j]+\E[\partial^2_{x_k}\log \mu_t(X_t)b^i_j(X_t)\mid X_t^j=x^j]\\
        &+\E[b^i_j(X_t) |\partial_{x_k} \log \mu_t(X_t)|^2 \mid X^j_t=x^j].
    \end{align*}
    We can notice that $\partial_{x_k}^2 \log \mu_t^j(x)=\E[\partial_{x_k}^2 \log \mu_t(X_t)\mid X^j_t]+\E[|\partial_{x_k} \log \mu_t(X_t)|^2\mid X_t]-|\partial_{x_k}\log\mu_t^j(x)|^2$; the result follows by repeatedly applying Young's and Jensen's inequalities, integrating in time and using Proposition \ref{pr: second derivatives square moments}. Local integrability follows by the fact that $\mu_t$ is locally bounded away from zero.    
\end{proof}

\subsection{Higher regularity and conclusion}
Given the higher regularity just obtained we can keep pushing for control over higher derivatives.

We have now obtained a significant control over the first and second derivatives of $\mu_t$ and $f_t$, which in turn allows us to control uniformly the derivatives of the mollified densities $\mu_t^r,\mu_t^{r,q}$ and the related coefficients. We summarize the properties that we will need going forward in the next lemma, which we prove in Appendix \ref{Appendix: mollification estimates}.
\begin{lemma}
\label{lm: bootstrapping} The following holds.
    \begin{enumerate}
        \item For each $\ell\geq q>0$, 
        \[\int_\ell^T \int_{\R^{dm}}|\partial_i^2 f_t^{r,q}(x)|^2\mu_t^{r,q}(x)\,dx\,dt\leq C<\infty,\]
        where $C$ does not depend on $r,q$ or $\ell$.
        \item For each $\ell\geq q >0$, $r>0$ and every $i$,
        \[\int_\ell^T\int_{\R^{dm}}\left( |\partial_i \log \mu_t^{r,q}(x)|^4+|\partial_i \log \mu_t^{r,q}(x)\partial_i \nabla \log \mu_t^{r,q}(x)|^2\right)\mu_t^{r,q}(x)\,dx\,dt\leq C<\infty\]
        Where the constant $C$ does not depend on the mollification parameters. The same holds with $\mu_t^r$ in place of $\mu_t^{r,q}$.
    \end{enumerate}
\end{lemma}

We still need more integrability from the second derivatives. We choose to obtain it by showing that the density has well behaved third (weak) derivatives.

\begin{proposition}
    For every $T>0$ and $i$,
\begin{align}
\label{eq: bound on third derivatives}
\sup_{r>0}\,\int_0^T \int_{\R^{dm}}|\partial_i^2\nabla \log \mu^r_t(x)|^2\,d\mu_t^r(x)\,d t<\infty.
\begin{split}
\end{split}
    \end{align}
\end{proposition}
\begin{proof}
    We compute 
    \begin{align*}
        \partial^2_i \partial_t \log \mu_t^{r,q}&=-\partial^2_i\text{div} f^{r,q}_t - \partial^2_i\nabla \log \mu^{r,q}_t \cdot f^{r,q}_t-2\partial_i \nabla \log \mu^{r,q}_t \cdot \partial_i f^{r,q}_t-\nabla \log \mu^{r,q}_t \cdot \partial_i^2 f^{r,q}_t \\
        &+\sigma^2 \partial^2_i\Delta \log \mu^{r,q}_t + 2\sigma^2 \partial^2_i \nabla \log \mu^{r,q}_t\cdot \nabla \log \mu_t^{r,q}+2\sigma^2|\partial_i \nabla \log \mu^{r,q}_t|^2.
    \end{align*}
    We use again Leibniz's rule to pass the derivative under the integral and find
 \begin{align*}
        \frac{d}{dt}&\int_{\R^{dm}} |\partial_i ^2 \log \mu_t^{r,q}(x)|^2\,d\mu_t^{r,q}(x)=-\int_{\R^{dm}} |\partial_i ^2 \log \mu_t^{r,q}(x)|^2\text{div}\left(\mu_t^{r,q}(x) ( f_t^{r,q}(x)-\sigma^2\nabla \log \mu_t^{r,q}(x))\right)\,dx\\
        &-2\int_{\R^{dm}} \partial_i^2\log\mu_t^{r,q}(x) \left(\partial_i^2 \text{div}f_t^{r,q}(x)+\partial_i^2\nabla\log\mu_t^{r,q}(x)\cdot f_t^{r,q}(x)\right)\,d\mu_t^{r,q}(x)\\
        &-2\int_{\R^{dm}} \partial_i^2\log\mu_t^{r,q}(x)\left(2\partial_i\nabla\log\mu_t^{r,q}(x)\partial_if_t^{r,q}(x)+\nabla \log\mu_t^{r,q}(x)\cdot \partial^2_i f_t^{r,q}(x)\right)\,d\mu_t^{r,q}(x)\\
        &+2\sigma^2 \int_{\R^{dm}} \partial_i^2 \log \mu_t^{r,q}(x) (\partial_i \Delta\log\mu_t^{r,q}(x)+2\partial_i^2\nabla \log \mu_t^{r,q}(x)\cdot\nabla\log\mu^{r,q}_t(x)+2|\partial_i\nabla\log\mu_t^{r,q}(x)|^2)\,d\mu_t^{r,q}(x).
    \end{align*}
Integrate by parts the first divergence and cancel out the resulting terms to obtain
    \begin{align*}
        \frac{d}{dt}&\int_{\R^{dm}} |\partial_i ^2 \log \mu_t^{r,q}(x)|^2\,d\mu_t^{r,q}(x)=-2\int_{\R^{dm}} \partial_i^2\log\mu_t^{r,q}(x) \left(\partial_i^2 \text{div}f_t^{r,q}(x)+\nabla \log\mu_t^{r,q}(x)\cdot \partial^2_i f_t^{r,q}(x)\right)\,d\mu_t^{r,q}(x)\\
        &-4\int_{\R^{dm}} \partial_i^2\log\mu_t^{r,q}(x)\partial_i\nabla\log\mu_t^{r,q}(x)\cdot\partial_i f_t^{r,q}(x)\,d\mu_t^{r,q}(x)\\
        &+2\sigma^2 \int_{\R^{dm}} \partial_i^2 \log \mu_t^{r,q}(x) (\partial_i \Delta\log\mu_t^{r,q}(x)+\partial_i^2\nabla \log \mu_t^{r,q}(x)\cdot\nabla\log\mu^{r,q}_t(x))\,d\mu_t^{r,q}(x)\\
&+4\sigma^2\int_{\R^{dm}}\partial_i \log\mu_t^{r,q}(x)|\partial_i\nabla\log\mu_t^{r,q}(x)|^2\,d\mu_t^{r,q}(x)\\
&=\text{Line}_1+\text{Line}_2+\text{Line}_3+\text{Line}_4.
\end{align*}
    We perform various integration by parts line by line. We start from Line$_1$: integrate by parts the $\text{div} f_t^{r,q}$ term and use Young's inequality
    \begin{align}
    \label{eq: line 1}
    \begin{split}
        \text{Line}_1&=
        -\int_{\R^{dm}} \partial_i^2 \log \mu_t^{r,q}(x)(\partial^2_i\text{div} f^{r,q}_t(x)+\nabla \log \mu^{r,q}_t(x) \cdot \partial^2_i f^{r,q}_t(x))\,d\mu_t^{r,q}(x)\\
        &=\int_{\R^{dm}} \partial^2_i \nabla \log \mu_t^{r,q}(x) \cdot \partial_i^2 f_t^{r,q}(x)\,d\mu_t^{r,q}(x)\\
        &\leq \frac \zeta 2\int_{\R^{dm}} |\partial^2_i\nabla \log \mu_t^{r,q}(x)|^2\,d\mu_t^{r,q}(x)+\frac{1}{2\zeta}\int_{\R^{dm}} |\partial_i^2 f_t^{r,q}(x)|^2\,d\mu_t^{r,q}(x).
    \end{split}
        \end{align}
    Integrate by parts the $\partial_i^2 \log\mu_t^{r,q}$ to see
    \begin{align*}
     \text{Line}_2&=-\int_{\R^{dm}} \partial_i^2 \log \mu_t^{r,q}(x) \partial_i \nabla \log \mu_t(x)\cdot \partial_i f_t^{r,q}(x)\,d\mu_t^{r,q}(x)\\
        &=\int_{\R^{dm}} \partial_i \log \mu_t^{r,q}(x)\partial^2_i \nabla  \log \mu_t^{r,q}(x) \cdot \partial_i f_t^{r,q}(x)+\partial_i \log \mu_t^{r,q}(x) \partial_i \nabla \log \mu_t^{r,q}(x) \cdot  \partial^2_i f_t^{r,q}(x)\,d\mu_t^{r,q}(x)\\
        &+\int_{\R^{dm}} \partial_i \log \mu_t^{r,q}(x) \partial_i \nabla \log \mu_t^{r,q}(x) \cdot \partial_i f_t^{r,q}(x) \partial_i \log\mu_t^{r,q}(x) \,d\mu_t^{r,q}(x).
    \end{align*}
    We can now use Young's inequality repeatedly to obtain
    \begin{align}
    \label{eq: line 2}
    \begin{split}
        \text{Line}_2&\leq \int_{\R^{dm}} \left(\frac{\zeta}{2}|\partial_i^2\nabla \log \mu_t^{r,q}(x)|^2 + \frac{1}{4\zeta} |\partial_i \log \mu_t^{r,q}(x)|^4 + \frac{1}{4\zeta} |\partial_i f_t^{r,q}(x)|^4 \right)\,d\mu_t^{r,q}(x)\\
        &+\frac 12 \int_{\R^{dm}} |\partial_i^2 f_t^{r,q}(x)|^2 + |\partial_i \log \mu_t^{r,q}(x)|^2 |\partial_i \nabla \log \mu_t^{r,q}(x)|^2\,d\mu_t^{r,q}(x)\\
        &+\frac 12 \int_{\R^{dm}}|\partial_i \log \mu_t^{r,q}(x) |^2 |\partial_i \nabla \log \mu_t^{r,q}(x)|^2 +\frac 12 |\partial_i f_t^{r,q}(x)|^4+\frac 12|\partial_i \log \mu_t^{r,q}(x)|^4\,d\mu_t^{r,q}(x)
    \end{split}
    \end{align}
     In the same way, integrate by parts the $\Delta\log\mu_t^{r,q}$ term to get
        \begin{align}
        \label{eq: line 3}
        \begin{split}
            \text{Line}_3=-\sigma^2 \int_{\R^{dm}} |\partial^2_i \nabla \log \mu_t^{r,q}(x)|^2\,d\mu_t^{r,q}(x).
    \end{split}
    \end{align}
    By integrating by parts $\partial_i^2\log \mu_t^{r,q}$ we find
    \begin{align}
    \label{eq: line 4}
    \begin{split}
\text{Line}_4&= 2\sigma^2\int_{\R^{dm}}\partial_i^2 \log \mu_t^{r,q}(x)|\partial_i \nabla \log \mu^{r,q}_t(x)|^2 \,d\mu_t^{r,q}(x)\\
        &=-2\sigma^2 \int_{\R^{dm}} 2\partial_i \log \mu_t^{r,q}(x) \partial^2_i \nabla \log \mu_t^{r,q}(x)\cdot \partial_i \nabla\log\mu_t^{r,q}(x)+ |\partial_i \log \mu_t^{r,q}(x)|^2 |\partial_i \nabla \log \mu_t^{r,q}(x)|^2\,d\mu_t^{r,q}(x)\\
        &\leq  4\sigma^2\int_{\R^{dm}} \frac \zeta 2| \partial^2_i \nabla \log \mu_t^{r,q}(x)|^2+\left(\frac{1}{2\zeta}+1\right)|\partial_i \log \mu_t^{r,q}(x)|^2|\partial_i \nabla \log \mu_t^{r,q}(x)|^2\,d\mu_t^{r,q}(x).
    \end{split}
    \end{align}
    Summing up \eqref{eq: line 1},\eqref{eq: line 2},\eqref{eq: line 3} and \eqref{eq: line 4}, integrating in time and taking $\zeta$ to be small enough we get for some constants $C_1,C_2>0$ that depend on $\sigma$ only:
    \begin{align*}
        &\int_{\R^{dm}} |\partial_i^2 \log \mu_t^{r,q}(x)|^2\,d\mu_t^{r,q}(x)+ C_1 \int_\ell^T \int_{\R^{dm}} |\partial^2_i\nabla \log \mu_t^{r,q}(x)|^2\,d\mu_t^{r,q}(x)\,dt \leq \int_{\R^{dm}} |\partial_i^2 \log \mu_\ell^{r,q}(x)|^2\,d\mu_\ell^{r,q}(x)\\
        &+C_2\int_\ell^T\int_{\R^{dm}} |\partial_i^2 f_t^{r,q}(x)|^2+|\partial_i \log \mu_t^{r,q} (x)|^4+|\partial_i f_t^{r,q}(x)|^4+|\partial_i \log \mu_t^{r,q} (x)\partial_i \nabla \log \mu_t^{r,q}(x)|^2\,d\mu_t^{r,q}(x)\,dt.
    \end{align*}
    The second line is bounded uniformly in $\ell,q,r$ by Lemma \ref{lm: bootstrapping}. We can now proceed in the same way as in Proposition \ref{pr: second derivatives square moments} to send $q,\ell \to 0$ and obtain
    \begin{align*}
    \int_{\R^{dm}}|\partial_i^2 \log \mu^r_T(x)|^2\,d\mu_T^r(x)+ C_1\int_0^T \int_{\R^{dm}}|\partial_i^2 \nabla \log \mu^r_t( x)|^2\,d\mu_t^r(x)\,dt\leq C_2+ \int_{\R^{dm}}|\partial_i^2 \log\mu^r_0(x)|^2\,d\mu^r_0(x).
\end{align*}
We check that the right hand side can be bounded indipendently of $r$: using Lemma \ref{lm: properties of mollified densities}(4) and Assumption \ref{assumption: initial regularity}(3)
\begin{align*}
    \int_{\R^{dm}}|\partial_i^2 \log\mu^r_0(x)|^2\,d\mu^r_0(x)&=\int_{\R^{dm}}\left|\frac{\partial_i ^2 \mu_0^r(x)}{\mu_0^r(x)}-|\partial_i \log \mu_0^r(x)|^2\right|^2\mu_0^r(x)\\
    &\leq 2 \int_{R^{dm}}\left|\frac{\partial_i ^2 \mu_0^r(x)}{\mu_0^r(x)}\right|^2+|\partial_i \log \mu_0^r(x)|^4 \,d\mu_0^r(x)\\
    &\leq 2 \int_{\R^{dm}}\left|\frac{\partial^2_i \mu_0(x)}{\mu_0(x)}\right|^2\,d\mu_0(x)+2\int_{\R^{dm}}|\partial_i \log \mu_0(x)|^4\,d\mu_0(x)<\infty
\end{align*}
It follows that for every $i$,
\begin{equation}
\label{eq: control over third derivaives r}\sup_{r>0}\,\int_0^T \int_{\R^{dm}}|\partial_i^2\nabla \log \mu^r_t(x)|^2\,d\mu_t^r(x)\,d t<\infty.
\end{equation}
\end{proof}
By controlling third derivatives of $\log\mu_t^r$, we can obtain control over the third moments of $\nabla^2\mu_t^r/\mu_t^r$, which is what we actually need to close the computations of Section \ref{section: intermediate to McKean--Vlasov}. These calculations are routine and thus relegated to Appendix \ref{Appendix: mollification estimates}.

\begin{proposition}
\label{pr: second derivatives third moments}
For every $T>0$,
\begin{align*}
    \sup_{r > 0 }\int_0^T\int_{\R^{dm}} &\left\|\frac{\nabla^2 \mu_t^r(x)}{\mu_t^r(x)}\right\|_{\textup{Fr}}^3\,d\mu_t^r(x)\,dt<\infty.
\end{align*}
\end{proposition}

We now finally prove Lemma \ref{lm: bound on time integral of relative entropy}; we make use of Propositions \ref{pr: second derivatives third moments} and \ref{pr: boundedness of Renyi Divergence}.
\begin{proof}[Proof of Lemma \ref{lm: bound on time integral of relative entropy}]
    The idea, as usual when dealing with mollification, is to interpolate between $\mu_t$ and $\mu_t \star_j K_h$; this is not straightforward since $\mu_t$ is only weakly differentiable. We thus work with $\mu_t^r$ and $ \mu_t^r \star_j K_h$, where $\mu_t^r$ solve \eqref{eq: PDE r}, and recover the desired bounds in the limit. From now on we will use the shorthand notation $\tilde \mu_t= \mu_t \star_j K_h$, defining $\tilde \mu_t^r$ analogously.\\
    
    Observe that by the data processing inequality and lower semicontinuity of relative entropy we have for every $t$
\begin{equation}
\label{eq: continuity of entropy}
    \limsup_{r\to 0}\ent(\mu_t^r\,\|\,\tilde\mu_t^r)\leq \ent(\mu_t\,\|\,\tilde \mu_t)\leq \liminf_{r \to 0} \ent(\mu_t^r\,\|\,\tilde\mu_t^r);
\end{equation}
Hence for every $t,$ $\ent(\mu^r_t\,\|\,\tilde \mu_t^r)\to \ent(\mu_t\,\|\,\tilde\mu_t)$. We now bound $\ent(\mu_t^r\,\|\,\tilde\mu_t^r)$ and use Fatou's lemma to pass to the limit. Recall the definition of the Chi Square divergence between two probability measures 
\begin{align*}
    \chi^2 (\alpha\,\|\,\beta)= \begin{cases}
        \int |\frac{d\alpha}{d\beta}-1|^2\,d\beta \text{ if }\alpha \ll \beta;\\
        +\infty \text{ otherwise.}
    \end{cases}
\end{align*}
By the classical information theoretic inequality $\ent \leq \chi^2$ (see for example \cite[Theorem 5]{Gibbs2002Review}), we find by interpolation (keeping in mind that $\mu_t^r \in C^\infty$),
\begin{align*}
     \ent(\mu_t^r\,\|\,\tilde\mu_t^r)\leq \chi^2 (\mu_t^r\,\|\,\tilde\mu_t^r)&= \int_{\R^{dm}} \frac{(\mu_t^r(x)-\tilde \mu_t^r(x))^2}{\tilde \mu_t^r(x)}\,dx\\
     &=\int_{\R^{dm}}\frac{1}{\tilde \mu_t^r(x)}\left\vert\int_{\R^{d}} \mu_t^r(x^j,x^{-j})-\mu_t^r(x^j-\sqrt hz,x^{-j})K(z)\,dz\right\vert^2\,dx\\
     &= h\int_{\R^{dm}}\frac{1}{\tilde \mu_t^r}\left\vert\int_{\R^{d}}\int_0^1 \nabla_{x_j}\mu_t^r(x^j-\sqrt h s z,x^{-j})\cdot zK(z)\,ds\,dz\right\vert^2\,dx
\end{align*}
Since $K_h$ is zero mean we can add $ \nabla_{x_j}\mu_t^r(x)\cdot z$ inside the integral and interpolate again to find
\begin{align*}  \ent(\mu_t^r\,\|\,\tilde\mu_t^r)&=h\int_{\R^{dm}}\frac{1}{\tilde \mu_t^r(x)}\left\vert\int_{\R^{d}}\int_0^1 \left(\nabla_{x_j}\mu_t^r(x^j-\sqrt h s z,x^{-j})-\nabla_{x_j} \mu_t(x)\right)\cdot zK(z)\,ds\,dz\right\vert^2\,dx\\
     &=h^2\int_{\R^{dm}}\frac{1}{\tilde \mu_t^r(x)}\left\vert\int_{\R^d}\int_0^1\int_0^1 z^\top \nabla_{x_j}^2  \mu_t^r(x^j-\sqrt h s \ell z,x^{-j})z\,d\ell\,ds\,K(z)\,dz\right\vert^2\,dx
\end{align*}
We now use Jensen's inequality to bring the square inside the inner integral, multiply and divide by $\mu_t(x-\sqrt h sz,y)$ and use Young's inequality to find
\begin{align*}
    &\ent (\mu_t^r\,\|\,\tilde\mu_t^r)\leq h^2\int_{\R^{dm}}\frac{1}{\tilde \mu_t^r(x)}\int_{\R^d}\int_0^1\int_0^1 \left\vert z^\top \nabla_{x_j}^2  \mu_t^r(x^j-\sqrt h s \ell z,x^{-j})z\right\vert^2\,d\ell\,ds\,K(z)\,dz\,dx\\
    & = h^2\int_{\R^{dm}}\int_{\R^d}\int_0^1\int_0^1\frac{\mu_t^r(x^j-\sqrt h s \ell z,x^{-j})}{\tilde \mu_t^r(x)}\\
    &\hspace{5cm}\times \left\vert\frac{ z^\top \nabla_{x_j}^2  \mu_t^r(x^j-\sqrt h s \ell z,x^{-j})z}{\mu_t^r(x^j-\sqrt h s \ell z,x^{-j})}\right\vert^2 \mu_t^r(x^j-\sqrt h s \ell z,x^{-j})\,d\ell\,ds\,K(z)\,dz\,dx\\
    &\leq \frac{h^2}{2} \int_{\R^{md}}\int_{\R^d}\int_{0}^1\int_0^1\left\vert\frac{ z^\top \nabla_{x_j}^2  \mu_t^r(x^j-\sqrt h s \ell z,x^{-j})z}{\mu_t^r(x^j-\sqrt h s \ell z,x^{-j})}\right\vert^3 \mu_t^r(x^j-\sqrt h s \ell z,x^{-j})\,d\ell\,ds\,K(z)\,dz\,dx\\
    &+\frac{h^2}{2} \int_{\R^{md}}\int_{\R^d}\int_{0}^1\int_0^1 \left|\frac{\mu_t^r(x^j-\sqrt h s\ell z,x^{-j})}{\tilde\mu_t^r(x)}\right|^3\mu_t^r(x^j-\sqrt h s z,x^{-j})K(z)\,ds\,dz\,dx\,.
\end{align*}
A change of variable allows us to rewrite the first element of the sum in a clearer way, and we can recognize the inner integral of the second term as $D_4(\mu_t^r \star_j \delta_{\sqrt h s \ell z}\,\|\,\tilde \mu^r_t)$. We obtain  
\begin{align*}
 \ent(\mu_t^r\,\|\,\tilde\mu_t^r)&\leq \frac{h^2}{2} \int_{\R^d}\int_{\R^{dm}}\left|\frac{z^\top \nabla_{x_j}^2 \mu_t^r(x)z}{\mu_t^r(x)}\right|^3\,d\mu_t^r(x)\,dK(z)\\
&+\frac{h^2}{2}\int_{\R^{d}}\int_0^1\int_0^1D_4(\mu_t^r\star_{j}\delta_{\sqrt h s \ell z}\,\|\,\tilde\mu_t^r)\,ds\,d\ell \,dK(z).
\end{align*}
The $D_4$ term can be bounded using Lemma \ref{pr: boundedness of Renyi Divergence}. As for the first term, using Assumption \ref{assumption: Kernel} we get
\begin{align*}
    \frac{h^2}{2} \int_{\R^d}\int_{\R^{dm}}\left|\frac{z^\top \nabla_{x_j}^2 \mu_t^r(x)z}{\mu_t^r(x)}\right|^3\,d\mu_t^r(x)\,dK(z)&\leq C\int_{\R^{dm}}\left\|\frac{\nabla^2_{x_j}\mu^r_t(x)}{\mu^r_t(x)}\right\|_{\textup{Fr}}^3\,d\mu_t^r(x).
\end{align*}
Putting all these bounds together and integrating in time, we obtain
\begin{align*}
    \int_0^T& \ent(\mu_t^r\,\|\,\tilde\mu_t^r)\,dt\leq h^2C\int_0^T\int_{\R^{dm}}\left\|\frac{\nabla^2_{x_j}\mu^r_t(x)}{\mu^r_t(x)}\right\|_{\textup{Fr}}^3\,d\mu_t^r(x)\,dt+h^2\int_0^T Ce^{6tM^2 \slash \sigma^2}\,dt\\
    &\leq C_{T} h^2.
\end{align*}
$C_T$ in the last inequality is a finite constant that does not depend on $r$. Boundedness of the term with second derivatives follows by assumption \ref{assumption: Kernel} and Proposition \ref{pr: second derivatives third moments}. We now use \eqref{eq: continuity of entropy} and Fatou's Lemma to send $h\to0$ and obtain
\begin{equation*}
    \int_0^T \ent(\mu_t\,\|\,\tilde\mu_t)\,dt \leq C_T h^2,
\end{equation*}
as required.
\end{proof}
It is likely that the estimate in Lemma \ref{lm: bound on time integral of relative entropy} is sharp since the typical size of the $L^2$-bias in nonparametric regression is $h$.
\bigskip

\begin{appendix}
\section{Well-Posedness of Conditional McKean--Vlasov equation}
\label{appendix: well posedness}
We prove the well-posedness of \eqref{eq: main SDE}. The result easily follows by techniques already known in the literature for more specialized settings. We consider the following definition of a weak solution, taken from \cite{PLD}
\begin{definition}
    A weak solution consists of a filtered probability space $(\Omega,\mathcal A, \mathcal F,\P)$ supporting a Wiener process $W_t=(W_t^i)_{i =1}^m$ of dimension $dm$, a continuous $\mathcal F$-adapted process $X_t$ and adapted bounded processes $Z_t^{i,j}$ such that
\begin{enumerate}
    \item The SDE holds:
    \begin{align*}
        dX_t^i&= \sum_{j=1}^m Z_t^{i,j}\,dt+V^i(X_t)\,dt+dW_t^i\,;\,i=1,\dots,m,\\
        \text{Law}(X_0)&=\mu.
    \end{align*}
    \item $Z_t^{i,j}=\E[b^i_j (X_t)\mid X_t^j]$ almost surely.
\end{enumerate}
\end{definition}
In what follows we will denote by $\mathcal C_T(\R^{dm})$ the space of continuous functions $f:[0,T]\mapsto \R^{dm}$.
\label{section: well posedness}
\begin{proof}[Proof of Lemma \ref{lm: well posedness MKV}]
    We essentially follow \cite[Proposition 3.8]{PLD} and \cite[Theorem 3.10]{hu2024htheorem}.\\
    We start by proving existence. By classical results for every $\mu_0$ that satisfies Assumption \ref{assumption: initial regularity}, the SDE
    \[dX_t^i= V^i(X_t)dt+\sqrt2\sigma\,dW_t^i\]
    initialized at $\text{Law}(X_0)=\mu_0$ has a unique solution, whose law we denote by $R$. Denote by $\Phi:\mathcal P(\mathcal C_T(\R^{dm}))\to  \mathcal P(\mathcal C_T(\R^{dm}))$ the function that associates to a law $P$ the unique law of the solution of 
    \[dX^i_t= \sum_{j=1}^m \E_P\left[b_j^i(X_t)\,\big \vert\, X_t^j\right]\,dt+V^i(X_t)\,dt+\sqrt 2 \sigma \,dW^i_t.\]
    The function $\Phi$ is well defined since the coefficients are bounded and the diffusion coefficient is constant, meaning that the above SDE has a weak solution that is unique in law for each $P$; clearly fixed points of $\Phi$ must solve \eqref{eq: main SDE}. Define the set $S$ as
    \[S=\left\{P\in \mathcal C_T(\R^{dm})\,\bigg\vert\, \left\|\frac{dP}{dR}\right\|_{L^2(P)}\leq M\right\}\]
    where $M$ is defined as
    \[M=\sup_{P \in \mathcal P(\mathcal C(\R^{dm}))}\left\|\frac{d\Phi(P)}{dR}\right\|_{L^2(\Phi(P))}\leq e^{2T\|b\|^2_\infty}.\] Notice that $S$ is weakly compact and convex and $\Phi(S)\subseteq S$. Define $K=\overline{\text{conv}}(\Phi(S))$ to be the closed convex hull of $\Phi(S)$; we obtain $K\subseteq S$, which implies that $K$ is weakly compact. We need to check that $\Phi$ is weakly continuous. By boundedness of the coefficients, we can use \cite[Corollary 6.4.3]{bogachev2022fokker} to check that every $P \in S$, the time marginals $\Phi(P)_t$ admit a density that is locally H\"older continuous in time and space (with the constants that can be taken to be independent of $P$); moreover, $\Phi(P)_0(0)=\mu_0(0)$ is clearly bounded uniformly in $P$. H\"older continuity and the uniform boundedness at $0$ are preserved by convex combinations, and thus we can take elements of $K$ to have uniformly bounded and equicontinuous densities. \\ 
    
    If we take $Q^n \in K^{\N}$ and assume $Q^n \rightharpoonup Q$, we can use Ascoli-Arzelà to check that $Q^n_t \to Q_t$ uniformly on compact sets of $[0,T]\times \R^{dm}$, and thus in total variation for every $t\in [0,T]$. Now using again standard entropy estimates (see \cite[Section 3]{SharpChaosLacker2023}), we get
    \begin{align*}
        H(\Phi(Q)[T]\,\|\,\Phi(Q^n)[T])&= \frac{1}{4\sigma^2}\int_0^T \sum_{i=1}^m\E_{\Phi(Q)}\big[\big|\sum_{i=1}^m \big(\E_{Q^n}[b_j^i(X_t)\mid X_t^j]-\E_{Q}[b_j^i(X_t)\mid X_t^j]\big)\big|^2\big]\,dt\\
        &\leq C\sum_{i,j}\int_0^T \E_{\Phi(Q)}[|\E_{Q^n}[b_j^i(X_t)\mid X_t^j]-\E_{Q}[b_j^i(X_t)\mid X_t^j]|^2]\,dt
    \end{align*}
    By \cite[Theorem 3.1]{ContinuityOfCE}, since $Q^n_t \to Q_t$ in total variation for every $t$, we know that $\E_{Q^n}[\E_{Q^n}[b_j^i(X_t)\mid X_t^j]\to \E_Q[\E_{Q^n}[b_j^i(X_t)\mid X_t^j]$ in $Q$-probability; since $Q \in K$ and $\Phi(Q)$ are mutually absolutely continuous, convergence is true also in $\Phi(Q)$-probability. Hence, by the Dominated Convergence Theorem, $H(\Phi(Q)[T]\,\|\,\Phi(Q^n)[T])\to 0$ and thus $\Phi(Q^n)\to \Phi(Q)$ in total variation (and in particular, weakly) by Pinsker's inequality. We have thus proved that $\Phi$ is a weakly continuous functional that maps a convex and compact set $K$ into itself. We can thus invoke Schauder's fixed point theorem to prove that $\Phi$ has a fixed point, which must thus be the law of a solution to \eqref{eq: main SDE}. \\

    To prove uniqueness in law, assume $\mu,\nu$ are the laws of two solutions of \eqref{eq: main SDE} and compute the relative entropy $\ent (\mu[T]\,\|\,\nu[T])$. We obtain
    \begin{align*}
        \ent (\mu[T]\,\|\,\nu[T])=\frac{1}{4\sigma^2}\sum_{i=1}^m\int_0^T \E_\mu\bigg[\Big|\sum_{j=1}^m\big(\E_\mu[b_j^i(X_t)\mid X_t^j]-\E_\nu[b_j^i(X_t)\mid X_t^j]\big)\Big|^2\bigg]\,dt
    \end{align*}
    bring the square inside the sum, use a total variation bound and Pinsker's inequality to obtain
    \begin{align*}
        \ent (\mu[T]\,\|\,\nu[T])\leq C\sum_{j=1}^m\int_0^T \E_\mu\Big[H(\mu_t^{X^j}\,\|\,\nu_t^{X^j})\Big]dt
    \end{align*}
    where we denote by $\mu_t^{x^j}$ the conditional density of $\mu_t$ given $X^j=x^j$; use the chain rule of relative entropy and the data processing inequality to obtain
    \begin{align*}
        \ent (\mu[T]\,\|\,\nu[T])\leq C\int_0^T H(\mu_t\,\|\,\nu_t)dt\leq C\int_0^T H(\mu[t]\,\|\,\nu[t])dt.
    \end{align*}    
   We can now invoke Gronwall's inequality to conclude the proof. 
\end{proof}
\section{Mollification Estimates}

\label{Appendix: mollification estimates}

This Section collects proofs of various lemmas and calculations related to the mollified densities $\mu_t^r,\mu_t^{r,q}$ and the related coefficients $f^r,f^{r,q}$ of the previous sections. We begin with some simple lemmas on technical details that have been used many times in Section \ref{sect: regularity estimates}.
\begin{lemma}
\label{lm: derivative under integral}
    Assume that for each ball $B\subset \R^d$ it holds
    \begin{align*}
        \int \int_B \left|\partial_x f(x,y)\right|+|f(x,y)|\,dx\,dy<\infty
    \end{align*}
    then in a weak sense
    \[\partial_x \int f(t,x)\,dt=\int \partial_x f(t,x)\,dt\]
\end{lemma}
\begin{proof}
Take a compactly supported $\varphi:\R^d \to \R$. Then 
\begin{align*}
    \int \int \varphi'(x) f(x,y)\,dy\,dx&=\int \int \varphi'(x)f(x,y)\,dx\,dy\\
    &=-\int \int \partial_x f(x,y)\varphi(x)\,dx\,dy=-\int \int \varphi(x)\partial_x f(x,y)\,dy\,dx
\end{align*}
We are using Fubini's theorem twice: in the first and the last equalities. We can do so thanks to the integrability assumptions.
    \end{proof}
\subsection{Estimates on Mollified Coefficients}
We now prove the lengthy Lemma \ref{lm: properties of mollified densities}.
\begin{proof}[Proof of Lemma \ref{lm: properties of mollified densities}]
    \begin{enumerate}
        \item We can check that for every $k \in \N$ by Jensen's inequality we have
        \begin{align*}
        \left|\frac{\nabla^k \mu_t^r}{\mu_t^r}\right|=\frac{1}{\mu_t^r}\left\vert\mu_t \star\left(\frac{ \nabla^k \eta^r}{\eta^r}\cdot\eta^r\right)\right\vert\leq \left\|\frac{\nabla^k\eta^r}{\eta^r}\right\|_\infty \leq C^k_r.
        \end{align*}
        As for the mollified in time density,
        \begin{align*}
        \left|\frac{\nabla^k \mu_t^{r,q}}{\mu_t^{r,q}}\right|&=\frac{1}{\mu_t^{r,q}}\left\vert\int_0^1\nabla^k\mu^r_{t-qs}(x) \theta(s)\,ds\right\vert=\frac{1}{\mu_t^{r,q}}\left\vert\int_0^1\frac{\nabla^k\mu^r_{t-qs}(x)}{\mu^r_{t-qs}(x)}\mu^r_{t-qs}(x) \theta(s)\,ds\right\vert \\
        &\leq \sup_{s\leq T}\left|\frac{\nabla^k\mu_s^r}{\mu_s^r}\right| \leq C^k_r.
        \end{align*}
        The bounds on logarithmic gradients follow by elementary calculations.
    \item Differentiability and boundedness follows from the definition of $f_t^r,f_t^{r,q}$. We can compute
    \begin{align*}
        |\partial_i f_t^r(x)|&= \left\vert\partial_i \log \mu_t^r(x) f_t^r(x) + \frac{1}{\mu_t^r(x)}\int_{\R^{dm}} f_t(z)\mu_t(z) \partial_i \eta^r(x-z)\,dz\right\vert\\
        &\leq |\partial_i \log \mu_t^r(x)f_t^r(x)|+\left\vert\frac{(f_t \mu_t)\star (\partial_i \log \eta^r \eta^r)(x)}{\mu_t^r(x)}\right\vert \leq C_r
    \end{align*}
    Where the bounds follows by boundedness of $f_t$ and the fact that $|\nabla^k \eta|\leq C_k \eta$. The computations for higher order derivatives and for $f_t^{r,q}$ are similar.
        \item     For the absolute continuity of $\mu_t^r$, see the discussion in \cite[Theorem 7.4.1]{bogachev2022fokker}. We can thus write for arbitrary $x\in \R^{dm}$ and $\bar s> \underline{s} \geq 0$:
\[\mu^r_{\bar s}(x)-\mu^r_{\underline{s}}(x)=\int_{\underline{s}}^{\bar{s}} \partial_t \mu_t^r(x)\,dt.\]
Where $\partial_t\mu^r_t$ is equal to the right side of \eqref{eq: PDE r}. We now differentiate in space to obtain
\[\partial_i\mu^r_{\bar s}(x)-\partial_i\mu^r_{\underline{s}}(x)=\int_{\underline{s}}^{\bar s} \partial_i\partial_t \mu_t^r(x)\,dt.\]
We can pass the derivative under the integral by Leibniz's rule, since $\partial_t \mu_t$ is differentiable in space with bounded derivative uniformly in  $t$. Hence $\partial_i \mu^r_t$ is absolutely continuous. To see absolute continuity of $\partial_i \partial_j\mu^r_t$ (and higher derivatives) repeat the above procedure, noticing that $\partial_i \partial_t \mu_t^r$ remains differentiable in space with bounded derivatives.
  \item In the weak sense, we know that $\partial_i^n\mu_t^r(x)= \int \partial_i^n \mu_t(x-rz)\eta(z)\,dz= [(\partial^n_i \mu_t/\mu_t) \mu_t] \star \eta^r$. We can then use Jensen's inequality to check for every $p\geq 1$,
    \[|\partial_i^n\mu_t^r/\mu_t^r|^p=\left\vert\frac{[(\partial^n_i \mu_t/\mu_t) \mu_t]\star \eta^r}{\mu_t^r} \right\vert^p\leq \frac{(|\partial_i^n \mu_t/\mu_t|^p \mu_t)\star \eta^r}{\mu_t^r};\]
integrate against $\mu_t^r$ to obtain the result.
    \item Using the previous point, Jensen's inequality and Fubini,
    \begin{align*}
        \int_\ell^T\int_{\R^{dm}} |\partial_i \log \mu_t^{r,q}(x)|^p\mu_t^{r,q}(x)\,dx\,dt&=
        \int_\ell^T\int_{\R^{dm}} \left|\frac{1}{\mu_t^{r,q}(x)}\int_{0}^1 \partial_i \log\mu_{t-sq}^r(x)\mu_{t-sq}^r(x)\theta(s)\,ds\right|^p\,d\mu_t^{r,q}(x)\,dt\\
        &\leq \int_\ell^T\int_{0}^1\theta(s)\int_{\R^{dm}}  \left|\partial_i \log \mu^r_{t-sq}(x)\right|^p\mu^r_{t-sq}(x)\,dx\,ds\,dt\\
        &\leq \int_{0}^1\theta(s)\int_\ell^T\int_{\R^{dm}}  \left|\partial_i \log \mu_{t-sq}(x)\right|^p\mu_{t-sq}(x)\,dx\,dt\,ds
    \end{align*}
    Where we are using point (4) for the last inequality to remove the space mollification. Now the change of variable $\bar t= t-sq$ and the assumption $\ell\geq q$ give
    \begin{align*}
        \int_\ell^T\int_{\R^{dm}} |\partial_i \log \mu_t^{r,q}(x)|^p\mu_t^{r,q}(x)\,dx\,dt&\leq \int_0^1 \theta(s)\int_{\ell-sq}^{T-sq}\int_{\R^{dm}}|\partial_i \log \mu_{\bar t}(x)|^p\mu_{\bar t}(x)\,dx\,d\bar t\,ds\\
        &\leq \int_0^1 \theta(s)\int_{0}^{T}\int_{\R^{dm}}|\partial_i \log \mu_{\bar t}(x)|^p\mu_{\bar t}(x)\,dx\,d\bar t\,ds\\
        &=\int_{0}^{T}\int_{\R^{dm}}|\partial_i \log \mu_{\bar t}(x)|^p\mu_{\bar t}(x)\,dx\,d\bar t.
    \end{align*}
 \item We want to differentiate $f_t^{r,q}$; we first need to differentiate $f_t^r$. To do so, we need to control the derivative $\partial_i (f_t\mu_t)$ when integrated against $\eta$. The fact that $\partial_i (f_t \mu_t)=\partial_i f_t\mu_t+f_t\partial_i \mu_t$ follows by the product rule for weak derivatives, the assumption that $\mu_t \in H^1(\R^d)$ and lemma \ref{Lm: derivatives of CE}(2). We can then verify that 
\begin{equation*}
    \int |f_t(\cdot-rz)\mu_t( \cdot-rz)|\,d\eta(z) \in L^1_\text{loc}(\R^d)\,;\, 
    \int |\partial_{i} \left(f_t( \cdot-rz)\mu_t( \cdot-rz)\right)|\,d\eta(z) \in L^1_\text{loc}(\R^d).
\end{equation*}
 Use Lemma \ref{lm: derivative under integral} to check that in the sense of weak derivatives
\begin{equation*}
    \partial_i \int f_t( x-rz)\mu_t( x-rz)\,d\eta(z)= 
    \int \left( \partial_i f_t( x-rz) \cdot \mu_t( x-rz) + \partial_i \mu_t( x-rz) \cdot f_t( x-rz)\right)\,d\eta(z).
\end{equation*}
With the above in mind, we can write the derivative of $f_t^r=\frac{(f_t \cdot \mu_t)\star \eta^r}{\mu_t^r}$ as
\begin{align*}
    \partial_i \frac{(f_t \mu_t)\star \eta^r(x)}{\mu^r(x)}&=\partial_i \frac{1}{\mu^r(x)}\int_{\R^{dm}} f_t(x-rz)\mu_t(x-rz)\eta(z)\,dz\\
    &=-\partial_i \log \mu^r(x) f_t^r(x)  + \frac{1}{\mu^r(x)}\int_{\R^{dm}} \partial_i f_t(x-rz)\mu_t(x-rz)\eta(z)\,dz\\
    &+\frac{1}{\mu^r(x)}\int_{\R^{dm}} f_t(x-rz)\partial_i\log \mu_t(x-rz)\mu_t(x-rz)\eta(z)\,dz.
\end{align*}
    It follows that almost everywhere,
    \begin{align*}
         \left\vert\partial_i f_t^r\right\vert&\leq  |\partial_i \log \mu^r \cdot f_t^r|+\bigg\vert\frac{(\partial_i f_t \cdot\mu_t)\star \eta^r}{\mu^r_t}\bigg\vert+\bigg\vert\frac{(f_t \cdot\partial_i \log \mu_t\cdot \mu_t)\star \eta^r}{\mu^r_t}\bigg\vert
    \end{align*}
We can integrate against $\mu_t^r$ and use Jensen's inequality, the boundedness of $f_t$ and Lemma \ref{Lm: derivatives of CE}(4) to obtain
\begin{align*}
\int_{\R^{dm}} |\partial_i f_t^r(x)|^p \mu_t^r(x)\,dx \leq C^p+C^p\int_{\R^{dm}}|\partial_i \log \mu_t(x)|^p\mu_t(x)\,dx     
\end{align*} 
where $C$ does not depend on $r$.
As for $\mu_t^{r,q}$, a similar line of reasoning brings us to
\begin{align*}
    \partial_i f_t^{r,q}= -\partial_i \log \mu_t^{r,q} f_t^{r,q}+\frac{1}{\mu_t^{r,q}}\int_0^1 \partial_i f_{t-sq}^r\mu^r_{t-sq}\theta(s)\,ds+\frac{1}{\mu_t^{r,q}}\int_0^1 f_{t-sq}^r \partial_i \log \mu_{t-sq}^r\mu_{t-sq}^r\theta(s)\,ds
\end{align*}
Since for every $x\in \R^{dm}$
\begin{equation*}
    \int_0^1\frac{\mu^r_{t-sq}\theta(s)\,ds}{\mu_t^{r,q}}=1,
\end{equation*}
we can use Jensen's inequality to find that for an arbitrary function $\varphi:[0,T]\times \R^d \to \R$ we have
\begin{align*}
    \left\vert\frac{(\varphi \cdot \mu^r)\star_t \theta^q}{\mu_t^{r,q}}\right\vert^p\leq \frac{(|\varphi|^p\cdot\mu^r)\star_t \theta^q}{\mu_t^{r,q}}
\end{align*}
And thus we obtain that for some constant $C$ that will change from line to line but that does not depend on $t,q$ or $r$, 
\begin{align*}\int_{\R^{dm}} \left|\partial_i f_t^{r,q}(x)\right|^p\,d\mu_t^{r,q}(x)&\leq 
    C^{p} \int_{\R^{dm}} \left|\partial_i \log \mu_t^{r,q}(x) f_t^{r,q}(x)\right|\,d\mu_t^{r,q}(x)\,dt
    \\
    &+C^{p} \int_{\R^{dm}} \left\vert\frac{1}{\mu_t^{r,q}(x)}\int_0^1 \partial_i f_{t-sq}^r(x)\mu^r_{t-sq}(x)\theta(s)\,ds\right|\,d\mu_t^{r,q}(x)\,dt\\
    &+C^{p}\int_{\R^{dm}} \left|\frac{1}{\mu_t^{r,q}}\int_0^1 f_{t-sq}^r(x) \partial_i \log \mu_{t-sq}^r(x)\mu_{t-sq}^r(x)\theta(s)\,ds\right|^p\,d\mu_t^{r,q}(x)\,dt\\
    &\leq C^{p}+C^p \int_0^1 \int_{\R^{dm}}(|\partial_i f_{t-sq}^r(x)|^p+|\partial_i\log\mu_{t-sq}^r(x)|^p) \mu_{t-sr}^{r}(x)\,dx \,\theta(s)\,ds\\
    &\leq C^p +C^p\int_0^1 \int_{\R^{dm}}|\partial_i \log\mu_{t-sq}( x)|^p\,d\mu_{t-sq}(x)\theta(s)\,ds
\end{align*}
By change of variables, Fubini and the previous inequality we obtain for an arbitrary $\ell\geq q$
\begin{align*}
    \int_{\ell}^T \int \left|\partial_i f_t^{r,q}(x)\right|^p\,d\mu_t^{r,q}(x)\,dt&\leq \int_\ell^T \int_0^1\int_{\R^{dm}} |\partial_i\log\mu_{t-sq}(x)|^p\,d\mu_{t-sq}(x)\theta(s)\,ds\,dt+C^p\\
    &=\int_0^1 \int_{\ell-sq}^{T-sq}\int_{\R^{dm}} |\partial_i\log\mu_{t}( x)|^p\,d\mu_t(x)\,dt\,\theta(s)\,ds+C^p\\
    &\leq  \int_{0}^{T} \int_{\R^{dm}}|\partial_i\log\mu_{t}( x)|^p\,d\mu_t(x)\,dt+C^p.
\end{align*}
    \end{enumerate}
\end{proof}
\subsection{Proof of Lemma \ref{lm: bootstrapping}}
We can now turn to the second significant Lemma of Section \ref{sect: regularity estimates}. We begin with an intermediate lemma.

\begin{lemma}
\label{lm: cross moments}
There exists a constant $C<\infty$ that depends on $\mu_0,T,\sigma,m,d$ and $f$ such that
\begin{align*}
     \int_0^T \int_{\R^{dm}}|\partial_i \log\mu_t^r (x) \partial_i \nabla\log\mu_t^r(x)|^2\,\mu_t^r(x)\,dx\,dt&\leq C.
\end{align*}
\end{lemma}
\begin{proof}
    We use equation \eqref{eq: HJB PDE first derivative} again to find in the sense of weak derivatives
    \begin{align*}
        \frac{d}{dt}\int_{\R^{dm}} &|\partial_i \log \mu^{r,q}_t(x)|^4\,d\mu_t^{r,q}(x) =
        -\int_{\R^{dm}}  |\partial_i \log \mu_t^{r,q}(x)|^4\text{div }(f^{r,q}_t(x)-\sigma^2 \nabla \mu_t^{r,q}(x))\,dx\\
        &-4\int_{\R^{dm}} (\partial_i \log \mu_t^{r,q}(x))^3 (\partial_i \text{div}f_t^{r,q}(x)+\partial_i \nabla \log \mu^{r,q}_t(x)\cdot f^{r,q}_t(x)+\nabla \log \mu^{r,q}_t(x) \cdot f^{r,q}_t(x))\,d\mu_t^{r,q}(x)\\
        &+4\sigma^2\int_{\R^{dm}} (\partial_i \log \mu^{r,q}_t(x))^3  (\partial_i \Delta \log \mu_t^{r,q}(x)+\partial_i |\nabla \log \mu_t^{r,q}(x)|^2)\,d\mu_t^{r,q}(x).
    \end{align*}
    Again we integrate by parts the divergence terms and the Laplacian:
\begin{align*}
     \frac{d}{dt}&\int_{\R^{dm}} |\partial_i \log \mu^{r,q}_t(x)|^4\,d\mu_t^{r,q}(x) =4\int_{\R^{dm}}  (\partial_i \log \mu_t^{r,q}(x))^3 \nabla \partial_i \log \mu_t^{r,q}(x)\cdot(f^{r,q}_t(x)-\sigma^2 \nabla \log\mu_t^{r,q}(x))\,d\mu_t^{r,q}(x)
     \\&+4\int_{\R^{dm}} 3|\partial_i \log \mu_t^{r,q}(x)|^2\partial_i \nabla \log \mu_t^{r,q}(x)\cdot \partial_i f_t^{r,q}(x)+ (\partial_i \log\mu_t^{r,q}(x))^3 \partial_i f_t^{r,q}(x)\cdot \nabla\log\mu_t^{r,q} \,d\mu_t^{r,q}(x)\\
    &-4\int_{\R^{dm}} (\partial_i \log \mu_t^{r,q}(x))^3 (\partial_i \nabla \log \mu_t^{r,q}(x)\cdot f_t^{r,q}(x)+\nabla \log \mu_t^{r,q}(x) \cdot f_t^{r,q}(x))\,d\mu_t^{r,q}(x)\\
    &-4\sigma^2 \int_{\R^{dm}} 3|\partial_i \log\mu_t^{r,q}(x)|^2 |\partial_i \nabla\log\mu_t^{r,q}(x)|^2+(\partial_i\log\mu_t^{r,q}(x))^3\partial_i \nabla\log\mu_t^{r,q}(x) \nabla\log\mu_t^{r,q}(x)\,d\mu_t^{r,q}(x)\\
        &+4\sigma^2\int_{\R^{dm}}(\partial_i\log\mu_t^{r,q}(x))^3\partial_i |\nabla \log \mu_t^{r,q}(x)|^2\,d\mu_t^{r,q}(x).
\end{align*}
We can check that the terms involving $(\partial_i \log\mu_t^{r,q})^3$ cancel out, yielding
\begin{align*}
    \frac{d}{dt}&\int_{\R^{dm}} |\partial_i \log \mu^{r,q}_t(x)|^4\,d\mu_t^{r,q}(x)\\
    &= 12\int_{\R^{dm}}|\partial_i \log \mu_t^{r,q}(x)|^2\partial_i \nabla \log \mu_t^{r,q}(x)\cdot \partial_i f_t^{r,q}(x)- |\partial_i \log\mu_t^{r,q}(x)|^2 |\partial_i \nabla\log\mu_t^{r,q}(x)|^2\,d\mu_t^{r,q}(x)\\
    &\leq 12\int_{\R^{dm}} \left(\frac{\zeta}{2}-\sigma^2\right)|\partial_i \nabla \log\mu_t^{r,q}(x)|^2 |\partial_i \log \mu_t^{r,q}(x)|^2+\frac{1}{2\zeta} |\partial_i f_t^{r,q}(x)|^2 |\partial_i \log \mu_t^{r,q}(x)|^2\,d\mu_t^{r,q}(x). 
\end{align*}
    Use Young's inequality again on the last term, take $\zeta$ to be small enough and integrate in time to find
    \begin{align}
    \begin{split}    
    \label{eq: control over fourth moment first derivative for mollification}
        \int_{\R^d} |\partial_i &\log \mu^{r,q}_t(x)|^4 \,d\mu_t^{r,q}(x)+ C\int_\ell^T \int_{\R^{dm}} |\partial_i \log \mu_t^{r,q}(x)|^2|\nabla \partial_i \log \mu_t^{r,q}(x)|\,d\mu_t^{r,q}(x)\,dt\\
        &\leq 
        \int_{\R^{dm}} |\partial_i \log \mu^{r,q}_\ell(x)|^4 \,d\mu_\ell^{r,q}(x)+C \int_\ell^T \int_{\R^{dm}} |\partial_i \log \mu_t^{r,q}(x)|^4+|\partial_i f_t^{r,q}(x)|^4\,d\mu_t^{r,q}(x)\,dt
        \end{split}
    \end{align}
    Notice that the last term is again bounded uniformly in $\ell,q,r$ by Lemma \ref{lm: properties of mollified densities}(6) and Lemma \ref{lm: bootstrapping}. We can now follow the same procedure as in Proposition \ref{pr: second derivatives square moments} to send $q,\ell\to 0$ and find
    \begin{align*}
    \begin{split}
        \int_{\R^{dm}} |\partial_i \log \mu^r_T(x)|^4\,d\mu^r_t(x)&+C \int_0^T\int_{\R^{dm}} |\partial_i \log \mu^r_t(x)|^2|\nabla \partial_i \log \mu^r_t(x)|^2\,d\mu^r_t(x)\,dt\\
        &\leq C+\int_{\R^{dm}} |\partial_i \log \mu^r_0(x)|^4 \,d\mu^r_0(x)\\
        &\leq C+\int_{\R^{dm}}|\partial_i \log \mu_0(x)|^4\,d\mu_0(x)<\infty.
    \end{split}
    \end{align*}
    Notice that we are using Lemma \ref{lm: properties of mollified densities}(4) in the last line.
\end{proof}

We can now perform the required calculations.

\begin{proof}[Proof of Lemma \ref{lm: bootstrapping}]
    \begin{enumerate}
       
    \item We know that 
        \begin{align*}
            \partial_i f_t^r(x)&=-\partial_i \log \mu^r(x) f_t^r(x)  + \frac{1}{\mu^r(x)}\int \partial_i f_t(x-rz)\mu_t(x-rz)\eta(z)\,dz\\
    &+\frac{1}{\mu^r(x)}\int  f_t(x-rz)\partial_i\log \mu_t(x-rz)\mu_t(x-rz)\eta(z)\,dz
        \end{align*}
        We follow the reasoning of Lemma \ref{lm: properties of mollified densities}: we know by Lemma \ref{Lm: derivatives of CE} and Propositions $\ref{pr: second derivatives square moments}$ and $\ref{pr: second derivatives of drift}$ that for almost every $t>0$, $f_t,\partial_i f_t,\partial_i^2 f_t$, 
        $\partial_i \log \mu_t,\partial_i^2 \log \mu_t \in L^2_\text{loc}(\R^d)$, and we can thus apply the product rule for weak derivatives and pass the derivative under the integral (as usual in the sense of weak derivatives). We can then compute
        \begin{align*}
            \partial_i^2 f_t^r&=-\partial_i^2 \log \mu_t^r\cdot f_t^r-\partial_i\log\mu_t^r \partial_i f_t^r-\partial_i\log \mu_t^r \cdot \frac{(\partial_i f_t \cdot \mu_t)\star \eta_r}{\mu_t^r}\\
            &+\frac{(\partial_i^2 f_t \cdot \mu)\star \eta_r}{\mu_t^r}+\frac{(\partial_i f_t\cdot \partial_i\log \mu_t\cdot \mu_t)\star \eta_r}{\mu_t^r}-\partial_i \log \mu_t^r\frac{(f_t \cdot \partial_i \log \mu_t\cdot \mu_t)\star \eta_r}{\mu_t^r}\\
            &+\frac{(\partial_i f \partial_i \log \mu_t \cdot \mu_t)\star \eta_r}{\mu_t^r}+\frac{(f_t \partial_i^2\log \mu_t\cdot \mu_t)\star \eta_r}{\mu_t^r}
        \end{align*}
        Notice that 
        \begin{align*}
            \partial_i^2 \log \mu_t^r&=\frac{\partial^2_i \mu_t^r}{\mu_t^r}- |\partial_i \log \mu_t^r|^2\\
            &=\frac{(\partial_i^2 \log \mu_t \cdot \mu_t)\star \eta^r}{\mu_t^r}+\frac{(|\partial_i \log \mu_t |^2\cdot \mu_t)\star \eta^r}{\mu_t^r}-|\partial_i \log \mu_t^r|^2.
        \end{align*}
        We can then use the fact that $f_t^r$ is uniformly bounded and Jensen and Young inequalities to obtain 
        \begin{equation}
        \label{eq: bound on partial_i^2 f^r}
            \int_{\R^{dm}}|\partial^2_i f_t^r(x)|^2\mu_t^r(x)\,dx\leq C \left(1+\E\left[|\partial_i^2 \log \mu_t(X_t)|^2+|\partial_i \log \mu_t(X_t)|^4+|\partial_i f_t(X_t)|^4+|\partial_i^2 f_t(X_t)|^2\right]\right).
        \end{equation}
         A similar computation yields
    \begin{align*}
            \partial_i^2 f_t^{r,q}&=-\partial_i^2 \log \mu_t^{r,q}\cdot f_t^{r,q}-\partial_i\log\mu_t^{r,q} \partial_i f_t^{r,q}-\partial_i\log \mu_t^{r,q} \cdot \frac{(\partial_i f^r_t \cdot \mu^r_t)\star_t \theta^q}{\mu_t^{r,q}}\\
            &+\frac{(\partial_i^2 f^r_t \cdot \mu_t^r)\star_t \theta^q}{\mu_t^{r,q}}+\frac{(\partial_i f^r_t\cdot \partial_i\log \mu^r_t\cdot \mu^r_t)\star \theta^q}{\mu_t^{r,q}}-\partial_i \log \mu_t^{r,q}\frac{(f^r_t \cdot \partial_i \log \mu^r_t\cdot \mu^r_t)\star_t \theta^q}{\mu_t^{r,q}}\\
            &+\frac{(\partial_i f_t^r \partial_i \log \mu^r_t \cdot \mu^r_t)\star_t \theta^q}{\mu_t^{r,q}}+\frac{(f^r_t \partial_i^2\log \mu^r_t\cdot \mu^r_t)\star_t \theta^q}{\mu_t^{r,q}}
    \end{align*}
    We can now follow the proof of Lemma \ref{lm: properties of mollified densities}(6): perform a change of variable, use Jensen's inequality and \eqref{eq: bound on partial_i^2 f^r} obtain 
    \begin{align*}
        \int_\ell^T &\int_{\R^{dm}} |\partial_i^2 f_t^{r,q}(x)|^2\,d\mu_t^{r,q}(x)\,dt\leq\\
        &C \int_0^T \int_{\R^{dm}} (1+|\partial_i^2 \log\mu_t(x)|^2+|\partial_i \log \mu_t(x)|^4+|\partial_i^2 f_t(x)|^2+|\partial_i f_t(x)|^4)\,d\mu_t(x)\,dt \leq C<\infty,
    \end{align*}
    where boundedness follows from Propositions \ref{pr: second derivatives of drift} and \ref{pr: second derivatives square moments} and Lemma \ref{lm: properties of mollified densities}(6).
    \item The proof of Lemma \ref{lm: cross moments} readily implies that $\sup_{t\leq T} \E[|\nabla\log\mu_t(X_t)|^4]<\infty.$
    We can then compute for some $\ell\geq q>0$, using the same reasoning as the proof of Lemma \ref{lm: properties of mollified densities}(4)
    \begin{align*}
        \int_{\R^{dm}} |\partial_i \log \mu^{r,q}_\ell(x)|^4\,d\mu_\ell^{r,q}(x)&\leq \int_0^1 \E[|\partial_i\log \mu_{\ell-sq}(X_{\ell-sq})|^4]\theta(s)\,ds\\
        &\leq \sup_{t\leq T} \E[|\partial_i\log\mu_t(X_t)|^4]<\infty.
    \end{align*}
    Now the result follows from Equation \eqref{eq: control over fourth moment first derivative for mollification}.
    \end{enumerate}
\end{proof}
\subsection{Some integration by parts}
We perform the routine integration by parts needed to get control over the third moments of $\nabla^2 \log \mu_t$.
\begin{proof}[Proof of Proposition \ref{pr: second derivatives third moments}]
    By integration by parts and Young's inequality we find
\begin{align*}\int_{\R^{dm}}&|\partial^2_i  \log \mu^r_t(x)|^3\,d\mu_t^r(x) = \int_{\R^{dm}} |\partial^2_i\log \mu_t^r(x)|\partial^2_i  \log \mu_t^r(x) \partial_i^2 \log \mu_t^r(x)\,d\mu_t^r(x)\\
    &\leq \int_{\R^{dm}} |\partial_i^2 \log \mu_t^r(x)\partial_i^3 \log \mu_t^r(x)\partial_i \log \mu_t^r(x)| + |\partial_i^2 \log \mu_t^r(x)\partial_i^2 \log \mu_t^r(x)|\partial_i \log \mu_t^r(x)|^2| \,d \mu_t^r(x)\\
    &+\int_{\R^{dm}} |\partial^3_i\log \mu_t^r(x)\partial_i^2 \log \mu_t^r(x) \partial_i \log \mu_t^r(x)|\,d \mu_t^r(x) \\
   &\leq \int_{\R^{dm}}|\partial_i^3 \log \mu_t^r(x)|^2\,d\mu_t^r(x)+\frac 32\int _{\R^{dm}}|\partial_i^2 \log \mu_t^r(x)|^2 |\partial_i \log \mu_t^r(x)|^2\,d\mu_t^r(x).
\end{align*}
The terms on the right hand side can be bounded using Lemma \ref{lm: cross moments} and \eqref{eq: bound on third derivatives}.
Integrate in time to find 
\begin{equation}
\label{eq: third moments second derivatives}
    \sup_{r>0}\,\int_0^T \int_{\R^{dm}} |\partial_i^2 \log \mu_t^r(x)|^3\,d\mu_t^r(x)\,dt< \infty
\end{equation}
A similar calculation gives us 
\begin{align*}
    \int_{\R^{dm}} |\partial_i \log \mu^r_t(x)|^6\,d\mu_t^r(x)&=-5\int_{\R^{dm}} |\partial_i \log \mu_t^r(x)|^4 \partial_i^2 \log \mu_t^r(x) \,d\mu_t^r(x)\\
    &\leq \frac{5}{3\zeta^3}\int_{\R^{dm}} |\partial_i^2\log\mu_t^r(x)|^3\,d\mu_t^r(x)+\frac{5\zeta^{3\slash 2}}{3}\int_{\R^{dm}} |\partial_i \log \mu_t^r(x)|^6\,d\mu_t^r(x) .
\end{align*}
Take $\zeta$ small enough, integrate in time and use \ref{eq: third moments second derivatives} to find
\begin{equation}
    \label{eq: sixth moments}\sup_{r>0}\,\int_0^T \int_{\R^{dm}} |\partial_i \log \mu_t^r(x)|^6\,d\mu_t^r(x)\,dt\leq C.
\end{equation}
We keep integrating by parts: consider $i\neq j$ and
\begin{align*}
\int_{\R^{dm}}&|\partial_i\partial_j\log\mu_t^r(x)|^3\,d\mu_t^r(x)\leq \int_{\R^{dm}} |\partial_j\log\mu_t^r(x)\partial_i \log \mu_t^r(x) |\partial_i\partial_j\log\mu_t^r(x)|^2|\,d\mu_t^r(x)\\
    &+\int_{\R^{dm}}|\partial_j \log\mu_t^r(x)\partial^2_i\partial_j\log\mu_t^r(x)\partial_i\partial_j\log\mu_t^r(x)|\,d\mu_t^r(x)\\
    &+ \int_{\R^{dm}}|\partial_j\log\mu_t^r(x) \partial_i^2\partial_j\log\mu_t^r(x)\partial_i\partial_j\log\mu_t^r(x)|\,d\mu_t^r(x)\\
    &\leq \frac12 \int_{\R^{dm}} |\partial_i\log\mu_t^r(x)|^2|\partial_i\partial_j\log\mu_t^r(x)|^2+|\partial_j\log\mu_t^r(x) \partial_i\partial_j\log\mu_t^r(x)|^2\,d\mu_t^r(x)\\
    &+\int|\partial_i^2\partial_j\log\mu_t^r(x)|^2\,d\mu_t^r(x)+\int |\partial_j\log\mu_t^r(x) \partial_i\partial_j \log\mu_t^r(x)|^2\,d\mu_t^r(x) 
\end{align*}
Integrate yet again in time, use Equations \eqref{eq: bound on third derivatives} and Lemma \ref{lm: cross moments} to find
\begin{align}
\label{eq: bound on third moments of second cross derivatives of log}
    \sup_{r>0}\int_0^T\int_{\R^{dm}} |\partial_i\partial_j \log \mu_t^r(x)|^3\,d\mu_t^r(x)<\infty
\end{align}
 Putting together \eqref{eq: sixth moments} and \eqref{eq: bound on third moments of second cross derivatives of log} and using Young's inequality, we conclude that there exists a constant independent of $r$ such that for each $i,j$,
\begin{align*}
    \int_0^T\int_{\R^{dm}} &\left|\frac{\partial_i\partial_j \mu_t^r(x)}{\mu_t^r(x)}\right|^3\,d\mu_t^r(x)\,dt\\
    &\leq 
    4\int_0^T\int_{\R^{dm}} \left|\partial_i \partial_j \log \mu_t^r(x)\right|^3+\left|\partial_i \log\mu_t^r(x)\partial_j \log \mu_t^r(x)\right|^3\,d\mu_t^r(x)\,dt
    \leq C<\infty.
\end{align*}
\end{proof}

\subsection{Boundedness of  p-Divergence}
In this subsection we check that condition \ref{assumption: initial regularity}(3) holds for every $t>0$ if we assume that it holds at $t=0$.
\begin{proposition}
\label{pr: boundedness of Renyi Divergence}
    Under assumption \ref{assumption: initial regularity}, for every $t\geq 0$ it holds
    \[\sup_{r>0}\int_{\R^{dm}}\int_0^1\int_0^1D_4(\mu_t^r\star_{j}\delta_{\sqrt h s \ell z}\,\|\,\tilde\mu_t^r)\,ds\,d\ell K(z)\,dz \,dt\leq e^{6tM/\sigma^2} C_D<\infty.\]
\end{proposition}
where $M$ is a constant that does not depend on $r$.
\begin{proof}
By convexity of $\D_4$, change of variable and symmetry of $K$ we obtain for arbitrary $u \in \R^d$
\begin{align*}
    \D_4(\mu_t^r\star_j \delta_{u}\,\|\,\tilde \mu_t^r)&\leq \int_{\R^d} \D_4(\mu_t^r \star_j\delta_{u}\,\|\,\mu_t^r\star_{j}\delta_{\sqrt h z})K(z)\,dz\\
    &=\int_{\R^d} \D_4(\mu_t^r \star_j\delta_{u+\sqrt h z}\,\|\,\mu_t^r)K(z)\,dz.
\end{align*}
Take $z \in \R^{d}$. Following the discussion before \eqref{eq: PDE r}, we obtain that $\mu_t^r\star_j \delta_z$ is a weak solution of the PDE
\[\partial_t \mu_t^r \star \delta_z = -\text{div}(f_t^r\star \delta_z \cdot \mu_t^r \star \delta_z)+\sigma^2 \Delta \mu_t^r\star \delta_z.\]
Now we can check that by Lemma \ref{lm: properties of mollified densities}(1),
\begin{align*}
    |\log \mu_t^r(x+u)-\log \mu_t^r(x)|= \left\vert\int_0^1 u\cdot \nabla\log\mu_t^r(x+su)\,ds\right\vert \leq \|\nabla\log\mu_t^r\|_{\infty}|u|;  
\end{align*}
which implies that 
\[\frac{\mu_t^r(x+u)}{\mu_t^r(x)}\leq e^{C|u|}.\]
hence we can use Lemma \ref{lm: derivative under integral} to differentiate under the integral and check that for an arbitrary $z \in \R^{d}:$
    \begin{align*}
    \frac{d}{dt}\D_4( \mu^r_t\star \delta_z\,\|\,\mu_t)&=\frac{d}{dt} \int_{\R^{dm}} \frac{( \mu^r_t(x-z))^4}{(\mu^r_t(x))^3}\,dx\\
    &=4\int_{\R^{dm}}\left(\frac{\mu^r_t(x-z)}{\mu_t^r(x)}\right)^{3}(-\text{div}(f^r_t(x-z) \mu^r_t(x-z)) +\sigma^2 \Delta \mu^r_t(x-z)\,dx \\
    &-3\int_{\R^{dm}} \left(\frac{ \mu^r_t(x-z)}{\mu_t^r(x)}\right)^4 (-\text{div}(f_t^r(x)\mu_t(x))+\Delta \mu^r_t(x) \sigma^2)\,dx\\
    &=  -\sigma^2 12\int_{\R^{dm}} \left|\nabla \log \frac{ \mu^r_t(x-z)}{\mu^r_t(x)}\right|^2 \left(\frac{ \mu^r_t(x-z)}{\mu^r_t(x)}\right)^{3} \mu^r_t(x-z)\,dx \\
    &+ 12\int_{\R^{dm}}(f_t^r(x-z) -f_t^r(x))\cdot \nabla \log \frac{ \mu^r_t(x-z)}{\mu^r_t(x)} \left(\frac{ \mu^r_t(x-z)}{\mu^r_t(x)}\right)^{3} \mu^r_t(x-z)\,dx\\
    &\leq \frac{12}{4\sigma^2}\int_{\R^{dm}} |f_t^r(x-z)-f_t^r(x)|^2 \left(\frac{ \mu^r_t(x-z)}{\mu^r_t(x)}\right)^{3} \mu^r_t(x-z)\,dx.
\end{align*}
Where we integrate by parts and the last inequality follows by Young's inequality. Since $|f_t^r\star\delta_z- f_t^r|^2\leq 2 \|f_t\|^2_\infty$, we obtain (in the sense of weak derivatives)
\[ \frac {d}{dt} D_4(\mu_t^r\star\delta_{z}\,\|\,\mu_t^r)\leq \frac{6\|f_t^r\|^2_\infty}{\sigma^2} \D_4( \mu^r_t\star \delta_z\,\|\,\mu_t^r).\]
and by Gronwall's lemma and the data processing inequality we conclude 
\[\D_4(\mu_t^r\star \delta_z \,\|\, \mu_t^r)\leq e^{6T M^2\slash \sigma^2}\D_4(\mu_0^r\star \delta_z\,\|\,\mu_0^r)\leq e^{6T M^2\slash \sigma^2} \D_4(\mu_0\star \delta_z\,\|\,\mu_0).\]
With $M= \|f\|_\infty$. Now we can use the data processing inequality and convexity of $D_4$ to check
\begin{align*}
\int_{\R^{d}}&\int_0^1\int_0^1D_4(\mu_t^r\star_{j}\delta_{\sqrt h s \ell z}\,\|\,\tilde\mu_t^r)\,ds\,d\ell K(z)\,dz\\
&\leq \int_{\R^{d}}\int_0^1\int_0^1 \D_4(\mu^r_t\star_j\delta_{\sqrt h s\ell z_1 +\sqrt h z_2}\,\|\,\mu_t^r)\,ds\,d\ell K(z_1)\,dz_1K(z_2)\,dz_2\,\\
&\leq e^{6tM/\sigma^2}\int_{\R^{d}}\int_0^1\int_0^1 \D_4(\mu_0\star_j\delta_{\sqrt h s\ell z_1 +\sqrt h z_2}\,\|\,\mu_0)\,ds\,d\ell K(z_1)\,dz_1K(z_2)\,dz_2\\
&=e^{6tM/\sigma^2} C_D<\infty,
\end{align*}
since $C_D<\infty$ by Assumption \ref{assumption: initial regularity}(3). 
\end{proof}

\end{appendix}

\bibliographystyle{plainnat}
\bibliography{Biblio}
\end{document}